\documentclass[review]{siamart}

\usepackage{latexsym}
\usepackage{dsfont}
\usepackage{amsfonts}
\usepackage{amssymb}
\usepackage{amsmath}
\usepackage{framed}
\usepackage{mathrsfs}
\usepackage{cite}
\usepackage{listings}
\usepackage{algorithmic,algorithm}
\usepackage{url}
\usepackage{textcomp}
\usepackage{pdfpages}
\usepackage{tabulary}
\usepackage{multirow}
\usepackage{makecell}
\usepackage{relsize}
\usepackage{mathtools}
\usepackage{subcaption}

\usepackage{multibib}
\newcites{bk,art,artno}{,,}

\newsiamremark{remark}{Remark}
\newsiamremark{example}{Example}
\newsiamthm{assumption}{Assumption}
\newsiamthm{definitionandlemma}{Definition \& Lemma}

\makeatletter
\newcommand*\bigcdot{\mathpalette\bigcdot@{.5}}
\newcommand*\bigcdot@[2]{\mathbin{\vcenter{\hbox{\scalebox{#2}{$\m@th#1\bullet$}}}}}
\makeatother

\newcommand\norm[1]{\left\lVert#1\right\rVert}

\newcommand\klam[1]{(#1)}

\newcommand*\al{\pmb \alpha}
\newcommand*\bet{\pmb \beta}
\newcommand*\gam{\pmb \gamma}
\newcommand*\del{\pmb \delta}
\newcommand*\eps{\pmb \epsilon}

\newcommand*\vphipmb{\pmb \varphi}
\newcommand*\psipmb{\pmb \psi}

\newcommand*\vpmb{\pmb{v}}
\newcommand*\wpmb{\pmb{w}}
\newcommand*\zpmb{\pmb{z}}

\newcommand*\upmb{\pmb{u}}
\newcommand*\xpmb{\pmb{x}}
\newcommand*\ypmb{\pmb{y}}

\newcommand*\gpmb{\pmb{g}}
\newcommand*\cpmb{\pmb{c}}
\newcommand*\nupmb{\pmb{\nu}}
\newcommand*\ybet{\pmb{y}_{\bet}}
\newcommand*\yplus{\pmb{y}_+}
\newcommand*\mbbc{\mathbb{C}}
\newcommand*\mbbn{\mathbb{N}}
\newcommand*\mbbr{\mathbb{R}}

\newcommand*\Ltwo{L^2(0,T;\mathbb{R}^M)}

\newcommand*\Astar{A_{\star}}
\newcommand*\Wcirc{\widehat{W}_{\circ}}
\newcommand*\alstar{\pmb \alpha_{\star}}
\newcommand*\alcirc{\pmb \alpha_{\circ}}

\newcommand*\mustar{\mu^{\star}}

\newcommand*\kertxt{\textnormal{ker}}
\newcommand*\imtxt{\textnormal{im}}
\newcommand*\spantxt{\textnormal{span}}
\newcommand*\ranktxt{\textnormal{rank}}

\newcommand*\obsmatzero{\mathcal{O}_N(C,A(\alcirc))}

\newcommand*\contmatzero{\mathcal{C}_N(A(\alcirc),B)}
\newcommand*\contmatzeroSHORT{\mathcal{C}_N^\circ}
\newcommand*\obsmatzeroSHORT{\mathcal{O}_N^\circ}
\DeclareMathOperator*{\argmax}{arg\,max}
\DeclareMathOperator*{\argmin}{arg\,min}

\newcommand{\comment}[1]{\textcolor{black}{#1}}
\definecolor{julien}{rgb}{1, 0, 0}
\definecolor{simon}{rgb}{0, 0.5, 0}

\newcommand{\TheTitle}{Gauss-Newton oriented greedy algorithms for the reconstruction of operators in nonlinear dynamics}
\newcommand{\TheAuthors}{S. Buchwald, G. Ciaramella and J. Salomon}

\headers{Operator reconstruction in nonlinear dynamics}{Buchwald, Ciaramella, Salomon}

\title{{\TheTitle}}
\author{S. Buchwald\thanks{Mathematics Department, Universit\"at Konstanz, Germany ({\tt simon.buchwald@uni-konstanz.de}).}
	\and
	G. Ciaramella\thanks{MOX, Dipartimento di Matematica, Politecnico di Milano ({\tt gabriele.ciaramella@polimi.it}).}
	\and 
	J. Salomon\thanks{ANGE team, INRIA Paris, France ({\tt julien.salomon@inria.fr}).}
}

\ifpdf
\hypersetup{
	pdftitle={\TheTitle},
	pdfauthor={\TheAuthors}
}
\fi

\begin{document}
\nolinenumbers
\maketitle

	% REQUIRED
\begin{abstract}
This paper is devoted to the development and convergence analysis of greedy reconstruction algorithms based on the strategy presented in [Y. Maday and J. Salomon, Joint Proceedings of the 48th IEEE Conference on
Decision and Control and the 28th Chinese Control Conference, 2009, pp. 375--379]. These procedures allow the design of a sequence of control functions that ease the  identification of unknown operators in nonlinear dynamical systems. 
The original strategy of greedy reconstruction algorithms is based on an offline/online decomposition of the reconstruction process and an ansatz for the unknown operator obtained by an a priori chosen set of linearly independent matrices. 
In the previous work [S. Buchwald, G. Ciaramella and J. Salomon, SIAM J. Control Optim., 59(6), pp. 4511-4537], convergence results were obtained in the case of linear identification problems. We tackle here the more general case of nonlinear systems. 
More precisely, \comment{we introduce a new greedy algorithm based on the linearized system. Then, we show that the controls obtained with this new algorithm} lead to the local convergence of the classical Gauss-Newton method applied to the online nonlinear identification problem. 
We then extend this result to the controls obtained on nonlinear systems where a local convergence result is also proved. The main convergence results are obtained for the reconstruction of drift operators \comment{in dynamical systems with linear and bilinear control structures}.
\end{abstract}
	
	% REQUIRED
	
	% REQUIRED

%%%%%%%%%%%%%%%%%%%%%%%%%%%%%%%%%%%%%%%%%%%%%%%%%%%%%%%%%%%%%%%%%%%%%%%%%%%%%%%%%%%%%
\section{Introduction}\label{sec:intro}

This paper is concerned with the development and the analysis of a new class of numerical methods for the operator reconstruction in controlled nonlinear differential systems.
The identification of unknown operators and parameters characterizing dynamical systems is a typical problem in several fields of applied science. 
In general, this is understood as an inverse problem, where the goal is to best fit simulated and experimental data. 
However, when a system is affected by input forces that can be controlled by an external user, the data used in the fitting process can be manipulated. 
If the input forces are not properly chosen, the fitting process can result in a very poor quality of the reconstructed parameters or operators.
Thus, it is natural to look for a set of input forces that allows one to generate good data permitting the best possible reconstruction.
This is a typical case in the field Hamiltonian identification in quantum mechanics \cite{Bonnabel2009,Badouin2008,Donovan2014,Fu_2017,Geremia2001,Geremia2003,Geremia2000,lebris:ham_id,Wang,Xue2019,Zhang2014,Zhou2012}, or in engineering in the context of state space realization \cite{DeSchutter2000,Juang1985,Sontag1998} and optimal design of experiments \cite{Atkinson2011,Bauer,Alexanderian2016}.

In this paper, we focus on the analysis and development of a class of greedy-type reconstruction algorithms (GR) that were introduced in \cite{madaysalomon} for Hamiltonian identification problems, further developed and analyzed in \cite{siampaper}, and later adapted to the identification of probability distributions for parameters in the context of quantum systems in \cite{spinpaper}.
This approach decomposes the identification process into an offline phase, where the control functions are computed by a GR algorithm, and an online phase, where the controls are used to generate experimental data to be used in an inverse problem for the final reconstruction of the unknown operator.
In \cite{siampaper}, a first detailed convergence analysis of this strategy was provided for the identification of the control matrix in a linear input/output system. Based on this analysis, the authors developed a new more efficient and robust numerical variant of the standard greedy reconstruction algorithm.
It was then shown in \cite{spinpaper} that this strategy is also able to reconstruct the probability distribution of control inhomogeneities for a spin ensemble in Nuclear Magnetic Resonance; see, e.g., \cite{glaser_training_2015,LibroQuantum}.
\comment{Even though the mentioned references focus also on GR algorithms for nonlinear problems, none of them presents a theoretical analysis for nonlinear problems.}

\comment{
When dealing with nonlinear problems it is common to consider linearization techniques. Classical methods for nonlinear problems, like the Gauss-Newton method \cite{kaltenbacher2008iterative}, are often derived from such linearizations.
Following this approach, we introduce a new class of GR algorithms, derived from the linearization of the dynamical system in a neighborhood of the unknown operator. We call these methods LGR (linearized greedy reconstruction algorithms) and present a detailed theoretical analysis. This is the first novelty of this work.
}

\comment{
Our analysis of LGR reveals that its hidden goal is to construct a set of control functions that attempt to make a specific matrix full rank in a neighborhood of the solution.
It turns out that this matrix is actually the Gauss-Newton (GN) matrix, namely the Jacobian of the nonlinear residual.
In this respect, the second main novelty of this paper is that we can interpret the procedure of our LGR as a process that computes controls making GN locally well defined and convergent, and thereby the online nonlinear reconstruction problem (locally) solvable. 
Thus, once the control functions are computed by LGR, GN can be used with guaranteed convergence.}
%Nevertheless, different variants of the Gauss-Newton method (or any other optimization algorithm) can be used in the online phase for solving the final identification problem.

\comment{
We provide a detailed analysis of LGR} for two classes of problems: the reconstruction of the drift matrix in linear input/output systems and the reconstruction of an Hamiltonian matrix in skew-symmetric bilinear systems. Both cases represent truly nonlinear problems, since the unknown operators act on the states of the systems.
Notice that the analysis that we are going to present for the drift matrix is also valid in the case of the reconstruction of the control matrix in a linear input/output systems, hence includes the case considered in \cite[Section 5]{siampaper}. Thus, this part of the present work is a substantial extension of the results of \cite{siampaper}.

\comment{
The development and analysis of LGR turns out to be also useful to obtain a first analysis of the original GR algorithm applied to nonlinear systems. This is the third main novelty of this paper, and it is achieved by relating the behaviors of GR and LGR: under appropriate controllability and observability assumptions, we show that the controls generated by GR are suitable also for LGR and thus make the GN Jacobian matrix full rank. The hypotheses of this analysis are discussed and studied for three classes of problems: the reconstruction of the drift matrix in linear input/output systems, the reconstruction of an Hamiltonian matrix in skew-symmetric bilinear systems, and the reconstruction of the control matrix in nonlinear dynamical systems with linear control structure.}

%We provide a detailed analysis for three classes of problems: the reconstruction of the drift matrix in linear input/output systems, the reconstruction of an Hamiltonian matrix in skew-symmetric bilinear systems, and the reconstruction of the control matrix in nonlinear dynamical systems with linear control structure. All these cases represent truly nonlinear problems, since either the unknown operators act on the states of the systems, or the state variable appears as argument of a nonlinear function.

%The second novelty of this work is to provide a first analysis of the original GR algorithm applied to nonlinear systems. This is achieved by relating the behavior of GR (applied to the original nonlinear problem) and LGR: under appropriate controllability and observability assumptions, we show that the controls generated by GR are suitable also for LGR and thus make the GN Jacobian matrix full rank.

The two GR and LGR approaches are compared by direct numerical experiments \comment{and by a global convergence analysis in a specific case.}
These show that GR and LGR are comparable when working locally near the solution. 
However, the GR applied directly to the original nonlinear system is superior when only poor information about the solution is available. 

The paper is organized as follows.
In Section \ref{sec:notation}, the notation used throughout this work is fixed.
\comment{Section \ref{sec:GeneralSetting} describes linearized systems and recalls GN for general reconstruction problems.
The LGR algorithm is introduced in Section \ref{sec:GRalgorithm}.}
In sections \ref{sec:LinearReconstruction} and \ref{sec:BilinearReconstruction}, we present analyses of LGR for the reconstruction of drift matrices in linear systems and Hamiltonian matrices in bilinear systems, respectively.
Section \ref{sec:nonlinear} focuses on GR for nonlinear problems, and
a corresponding analysis is provided in section \ref{sec:nonlinear:analysis}.
Within section \ref{sec:Improvements}, we recall and extend an optimized greedy reconstruction (OGR) algorithm introduced in \cite{siampaper}.
\comment{In section \ref{sec:global}, a global analysis is provided for a specific case of a system with bilinear control structure.}
The LGR, GR and OGR algorithms are then tested numerically in section \ref{sec:numerics}.
%Finally, our conclusions are drawn in Section \ref{sec:conclusion}.

\section{Notation}\label{sec:notation}
Consider a positive natural number $N$. 
We denote by $\langle{\bf v},{\bf w}\rangle:={\bf v}^\top{\bf w}$, for any
${\bf v},{\bf w}\in \mathbb{R}^N$, the usual real scalar product on $\mbbr^N$, and by
$\norm{\, \cdot \,}_2$ the corresponding norm. For any $A \in \mathbb{R}^{N \times N}$, 
$[A]_{j,k}$ is the $j,k$ (with $j,k \leq N$) entry of $A$, and
the notation $A_{[1:k,1:j]}$ indicates the upper left submatrix of $A$ of size $k \times j$,
namely, $[A_{[1:k,1:j]}]_{\ell,m}:=[A]_{\ell,m}$ for $\ell=1,\dots,k$ and $m=1,\dots,j$.
Similarly, $A_{[1:k,j]}$ denotes the column vector in $\mathbb{R}^k$ corresponding to the
first $k$ elements of the column $j$ of $A$.
Additionally, $\imtxt(A)$ is the image of $A$, and $\kertxt(A)$ its kernel.
We indicate by $\mathfrak{so}(N)$ the space of skew-symmetric matrices in $\mbbr^{N\times N}$. 
Moreover, when talking about symmetric matrices, PD and PSD stand for positive definite and semidefinite, respectively.
By $(A,B,C)$ we denote the input/output dynamical system
\begin{equation}\label{eq:IOSyst}
	\xpmb(t)=C\ypmb(t),\quad\dot{\ypmb}(t)=A\ypmb(t)+B\eps(t),\quad \ypmb(0)=\ypmb^0.
\end{equation}
\comment{The control functions $\eps$ belong to the admissible set $E_{ad}$, which is assumed to be a non-empty and weakly compact subset of $\Ltwo$, $M\in\mbbn$.
%some Hilbert space of functions from $[0,T]$ to $\mbbr^M$, $M\in\mbbn$ (e.g., $E_{ad} \subset \Ltwo$). 
This guarantees well posedness of the optimal control problems considered in this work.
Moreover, we assume that $E_{ad}$ contains $\eps\equiv0$ as an interior point.\footnote{This is a reasonable assumption, since it is satisfied, e.g., for standard box constraints, which are quite often used in the applications.}
This hypothesis is used in our analysis in sections \ref{sec:LinearReconstruction}, \ref{sec:BilinearReconstruction} and \ref{sec:nonlinear}.}
For an interval $X\subset\mbbr$, the notation $\phi:X\rightrightarrows\mbbr^N$ indicates that $\phi$ is a set-valued correspondence, i.e. $\phi(x)\subset\mbbr^N$ is a set for $x\in X$.
Finally, we denote by $\mathcal{B}^N_r(x)\subset\mbbr^N$ the $N$-dimensional ball with radius $r>0$ and center $x\in\mbbr^N$.

	%%%%%%%%%%%%%%%%%%%%%%%%%%%%%%%%%%%%%%%%%%%%%%%%%%%%%%%%%%%%%%%%%%%%%%%%%%%%%%%%%%%%%
\section{\comment{Linearized systems and Gauss-Newton method (GN)}}\label{sec:GeneralSetting}
Consider a state $\ypmb(t)\in\mbbr^N$, $N\in\mbbn$, whose time evolution is governed by the system of ordinary differential equations (ODE)
\begin{equation}\label{eq:ODE}
	\dot{\ypmb}(t)=f(\Astar,\ypmb(t),\eps(t)), \; t \in(0,T], \quad\ypmb(0)=\ypmb^0,
\end{equation}
where $\ypmb^0\in\mathbb{R}^N$ is the initial state and $\eps\in E_{ad}$ denotes a control function belonging to $E_{ad}$.
The operator $\Astar$ is unknown and assumed to lie in the space spanned by a finite-dimensional set $\mathcal{A}=\{A_1,\ldots,A_K\}$, $K\in\mathbb{N}$, and we write $\Astar = \sum_{j=1}^{K}\al_{\star,j}A_j=:A(\alstar)$.
We assume that $f:\spantxt(\mathcal{A})\times\mbbr^N\times \mbbr^M\rightarrow\mbbr^N, (A,\ypmb,\eps)\mapsto f(A,\ypmb,\eps)$ is differentiable in $A$ and $\ypmb$.

To identify the unknown operator $\Astar$ one uses a set of control functions $(\eps^m)_{m=1}^{K}\subset E_{ad}$ to perform $K$ laboratory experiments and obtain the experimental data
\begin{equation}\label{eq:laboratorydata}
	\vphipmb^\star_{data}(\eps^m):=C\ypmb(\Astar,\eps^m;T),\textnormal{ for }m=1,\ldots,K.
\end{equation}
Here, $\ypmb(\Astar,\eps;T)$ denotes the solution to \eqref{eq:ODE} at time $T>0$, corresponding to the operator $\Astar$ and a control function $\eps$. The matrix $C\in\mathbb{R}^{P\times N}$ ($P\leq N$) is a given observer matrix. The measurements are assumed not to be affected by noise.

Using the set $(\eps^m)_{m=1}^{K}$ and the data $(\vphipmb^\star_{data}(\eps^m))_{m=1}^{K}\subset\mbbr^P$, the unknown vector $\al$ is obtained by solving the least-squares problem
\begin{equation}\label{eq:identificationproblem}
	\min_{\al\in\mbbr^K}\frac{1}{2}\sum_{m=1}^{K}\norm{\vphipmb^\star_{data}(\eps^m)-C\ypmb(A(\al),\eps^m;T)}_2^2.
\end{equation}
%Clearly $\alstar$ is a global minimizer for \eqref{eq:identificationproblem}.

GN is a typical iterative strategy to solve problems of the form \eqref{eq:identificationproblem}, and its
process is initialized by a vector which we will call $\al_{\circ}\in\mbbr^K$. 
We denote by $\al_c\in\mbbr^K$ the GN iterate, and define $f_m(\al):=\frac{1}{2}\sum_{i=1}^{P}\norm{(R_m(\al))_i}_2^2=\frac{1}{2}R_m(\al)^\top R_m(\al)$, where
\begin{align}
R_m(\al)&:=C\ypmb(A(\al),\eps^m;T)-\vphipmb^\star_{data}(\eps^m),\label{eq:rm}
%\\f_m(\al)&:=\frac{1}{2}\sum_{i=1}^{P}\norm{(R_m(\al))_i}_2^2=\frac{1}{2}R_m(\al)^\top R_m(\al),\label{eq:fm}
\end{align}
for $m\in\{1,\ldots,K\}$. Thus, the identification problem \eqref{eq:identificationproblem} is equivalent to
\begin{equation}\label{eq:GaussNewtonIdentProblem}
\min_{\al\in\mbbr^K}\sum_{m=1}^{K}f_m(\al).
\end{equation}
Given an iterate $\al_c$, GN computes the new iterate by solving a problem of the form
\begin{equation}\label{eq:GNStepProblemR}
\min_{\al\in\mbbr^K} \sum_{m=1}^{K}\|R_m'(\al_c)(\al-\al_c) - R_m(\al_c)\|_2^2,
\end{equation}
where $R_m'(\al_c)\in\mbbr^{P\times K}$ denotes the Jacobian of $R_m$ at $\al_c\in\mbbr^K$.
The first-order optimality condition of \eqref{eq:GNStepProblemR} is
\begin{equation}\label{eq:GNStepOptimality}
\sum_{m=1}^{K}\Big(R_m'(\al_c)^\top R_m'(\al_c)\Big)\al=\sum_{m=1}^{K}R_m'(\al_c)^\top R_m(\al_c),
\end{equation}
where $\sum_{m=1}^{K}R_m'(\al_c)^\top R_m'(\al_c)=:\widehat{W}_c\in\mbbr^{K\times K}$ is symmetric PSD.
Now, we recall the following convergence result from \cite[Theorem 2.4.1]{kelley}.
%%%%% SUPPL
% (for a proof see also the supplementary material \cite{supp_material}).
%%%%% SUPPL
\begin{lemma}[local convergence of GN]\label{lem:gaussnewton}
	Consider a problem of the form \eqref{eq:GaussNewtonIdentProblem}. Let $\alstar$ be a minimizer of \eqref{eq:GaussNewtonIdentProblem} such that for all $m\in\{1,\ldots,K\}$ the function $R_m$ is Lipschitz continuously differentiable near $\alstar$ and $R_m(\alstar)=0$.
	If the initialization vector $\alcirc\in\mbbr^K$ is sufficiently close to $\alstar$, and $\widehat{W}_c$ is PD for all iterates $\al_c\in\mbbr^K$, then GN converges quadratically to $\alstar$.
\end{lemma}
Lemma \ref{lem:gaussnewton} implies that, given an initialization vector $\alcirc$ sufficiently close to the solution $\alstar$, the functions $(\eps^m)_{m=1}^{K}$ should be chosen such that the GN matrix $\widehat{W}_c=\sum_{m=1}^{K}R_{\eps^m}'(\al_c)^\top R_{\eps^m}'(\al_c)$ is PD for all $\al_c\in\mbbr^K$ in a neighborhood of $\alstar$.
Notice that $\widehat{W}_c$ being PD is equivalent to \eqref{eq:GNStepProblemR}-\eqref{eq:GNStepOptimality} being uniquely solvable.
Using \eqref{eq:rm}, we can write \eqref{eq:GNStepProblemR} more explicitly. For a direction $\delta\al\in\mbbr^K$, we have
\begin{equation}\label{eq:RmPrimeC}
R_m'(\al_c)(\delta\al)=C\delta\ypmb_c(A(\delta\al),\eps^m;T),
\end{equation}
where $\delta\ypmb_c(A(\delta\al),\eps^m;T)$ denotes the solution at time $T$ to the linearized equation
\begin{equation}\label{eq:ODElinearized}
\resizebox{.9\hsize}{!}{$\begin{cases}
	\dot{\delta\ypmb}_c=\partial_{\ypmb}f(A(\al_c),\ypmb_c,\eps)\delta\ypmb_c+\sum_{j=1}^{K}\delta\al_j\Big(\partial_{A} f(A(\al_c),\ypmb_c,\eps)(A_j)\Big),\hspace*{0.1cm} \delta\ypmb_c(0)=0,\\
	\dot{\ypmb}_c=f(A(\al_c),\ypmb_c,\eps),\quad\ypmb_c(0)=\ypmb^0,
	\end{cases}$}
\end{equation}
at $\eps=\eps^m$.
Hence, problem \eqref{eq:GNStepProblemR} can be written as
\begin{equation}\label{eq:GNStepProblem}
\min_{\al\in\mbbr^K} \sum_{m=1}^{K}\|C\delta\ypmb_c(A(\al-\al_c),\eps^m;T) - R_{m}(\al_c)\|_2^2.
\end{equation}
Notice that the vectors $R_m(\al_c)\in\mbbr^P$ are independent of $\al$ and can therefore be considered as fixed data when solving \eqref{eq:GNStepProblem}.
Now, we recall that the GR algorithm, introduced in \cite{madaysalomon} and further analyzed in \cite{siampaper}, was designed specifically to generate control functions $(\eps^m)_{m=1}^{K}$ that make problems of the form \eqref{eq:GNStepProblem} uniquely solvable.

\section{A linearized GR algorithm (LGR)}\label{sec:GRalgorithm}
Let us assume to be provided with \comment{a vector $\alcirc$ sufficiently close to $\alstar$ (recall from section~\ref{sec:GeneralSetting} that $A_\star = A(\al_\star)$).} Further, let $\delta\ypmb_\circ(A(\al-\alcirc),\eps^m;T)$ denote the solution at time $T$ to
\begin{equation}\label{eq:ODElinearizedCirc}
\resizebox{.9\hsize}{!}{$\begin{cases}
	\dot{\delta\ypmb}_\circ=\partial_{\ypmb}f(A(\alcirc),\ypmb_\circ,\eps)\delta\ypmb_\circ+\sum_{j=1}^{K}(\al_j-\al_{\circ,j})\Big(\partial_{A} f(A(\alcirc),\ypmb_\circ,\eps)(A_j)\Big),\hspace*{0.1cm} \delta\ypmb_\circ(0)=0,\\
	\dot{\ypmb}_\circ=f(A(\alcirc),\ypmb_\circ,\eps),\quad\ypmb_\circ(0)=\ypmb^0.
	\end{cases}$}
\end{equation}
%Since we want \eqref{eq:GNStepProblem} to be uniquely solvable for $\al_c$ in a neighborhood of the unknown $\alstar$, we assume to be provided with an initialization vector $\alcirc$ for GN that is sufficiently close to $\alstar$.
The goal is to generate control functions $(\eps^m)_{m=1}^{K}$ such that
\eqref{eq:GNStepProblem} in $\alcirc$, that is
 \begin{equation}\label{eq:GNStepProblemCirc}
 \min_{\al\in\mbbr^K} \sum_{m=1}^{K}\|C\delta\ypmb_\circ(A(\al-\alcirc),\eps^m;T) - R_m(\alcirc)\|_2^2,
 \end{equation}
 is uniquely solvable.
Then, in Section \ref{subsec:PosDefGaussNewtonLinear}, we show that if \eqref{eq:GNStepProblemCirc} is uniquely solvable, the same holds for \eqref{eq:GNStepProblem} at all \comment{$\al$ in a neighborhood of $\alcirc$.
Thus, if $\alcirc$ is an initialization vector for GN, then \eqref{eq:GNStepProblem} is uniquely solvable for all iterates $\al_c$ of GN.}

\comment{
The set $(\eps^m)_{m=1}^{K}$ is computed by  Algorithm \ref{algo:General}. 
This is the original GR algorithm from \cite{madaysalomon} applied to \eqref{eq:GNStepProblemCirc}, which involves the linearized equations \eqref{eq:ODElinearizedCirc}. 
}
\begin{algorithm}[t]
	\caption{Linearized Greedy Reconstruction Algorithm (LGR)}
	\begin{algorithmic}[1]\label{algo:General} 
	\begin{small}
		\REQUIRE A set of linearly independent operators $\mathcal{A}=\{A_1,\ldots,A_{K}\}$. 
		Recall that $\delta\ypmb_\circ(A,\eps;T)$ solves~\eqref{eq:ODElinearizedCirc}.
		\STATE Compute the control $\eps^1$ by solving \begin{equation}\label{eq:InitializationGeneral}
		\max_{\eps\in E_{ad}} \norm{C\delta\ypmb_\circ(A_1,\eps;T)}_2^2.
		\end{equation}
		\FOR{ $k=1,\dots, K-1$ }
		\STATE \underline{Fitting step}: Let $A^{(k)}(\bet):=\sum_{j=1}^{k}\bet_jA_j$, find $\bet^k=(\bet^{k}_j)_{j=1,\dots,k}$ that solves \begin{equation}\label{eq:FittingStepGeneral}
		\min_{\bet\in\mbbr^{k}}\sum_{m=1}^{k}\norm{C\delta\ypmb_\circ(A^{(k)}(\bet),\eps^m;T)-C\delta\ypmb_\circ(A_{k+1},\eps^m;T)}_2^2.
		\end{equation}
		%where $A^{(k)}(\bet):=\sum_{j=1}^{k}\bet_jA_j$.
		\STATE \underline{Splitting step}: Find $\eps^{k+1}$ that solves \begin{equation}\label{eq:SplittingStepGeneral}
		\max_{\eps\in E_{ad}}\norm{C\delta\ypmb_\circ(A^{(k)}(\bet^{k}),\eps;T)-C\delta\ypmb_\circ(A_{k+1},\eps;T)}_2^2.
		\end{equation}
		%\STATE Update $k \leftarrow k+1$.
		\ENDFOR
	\end{small}
	\end{algorithmic}
\end{algorithm}
\setlength{\textfloatsep}{8pt}
\comment{For this reason, we call it LGR, namely Linearized Greedy Reconstruction algorithm.
As in the case of the original GR, the heuristics of LGR is that the set $(\eps^m)_{m=1}^{K}$ must allow distinguishing the states of the system \eqref{eq:ODElinearizedCirc} corresponding to any two matrices $A(\widehat{\al})$ and $A(\widetilde{\al})$, for $\widehat{\al} \neq \widetilde{\al}$.
Suppose that the first $k$ controls $(\eps^m)_{m=1}^{k}$ are already computed, the new control $\eps^{k+1}$ is obtained by solving first the fitting step problem \eqref{eq:FittingStepGeneral}, identifying two states that cannot be distinguished by the controls $(\eps^m)_{m=1}^{k}$, and then the splitting step problem \eqref{eq:SplittingStepGeneral}, computing $\eps^{k+1}$ with the goal of distinguishing these two states.
For more details see \cite{madaysalomon,siampaper}.}

Our goal is to prove that the set $(\eps^m)_{m=1}^{K}$ makes $\Wcirc:=\sum_{m=1}^{K}R_m'(\alcirc)^\top R_m'(\alcirc)$ PD, and thus \eqref{eq:GNStepProblemCirc} uniquely solvable.
From \eqref{eq:ODElinearizedCirc}, we have that $\delta\ypmb_\circ$ is linear in $\al$. Thus, $R_m'(\alcirc)(\delta\al)=\delta\ypmb_\circ(A(\delta\al),\eps^m;T)=\sum_{j=1}^{K}\delta\al_j C\delta\ypmb_\circ(A_j,\eps^m;T)$.
Hence, $R_m'(\alcirc)$ is a matrix with columns $R_m'(\alcirc)_j=C\delta\ypmb_\circ(A_j,\eps^m;T)$ for $j=1,\ldots,K$, and hence
\begin{equation}\label{eq:Wwidehatcirc_ij}
[\widehat{W}_\circ]_{i,j}=\sum_{m=1}^{K}\langle C\delta\ypmb_\circ(A_i,\eps^m;T),C\delta\ypmb_\circ(A_j,\eps^m;T)\rangle,\quad i,j\in\{1,\ldots,K\}.
\end{equation}
Using \eqref{eq:Wwidehatcirc_ij}, we can rewrite \eqref{eq:InitializationGeneral}, \eqref{eq:FittingStepGeneral} and \eqref{eq:SplittingStepGeneral} in a matrix form.
\begin{lemma}[Algorithm \ref{algo:General} in matrix form]\label{lem:ALGinWGeneral}
	Consider Algorithm \ref{algo:General}. Then:
	\begin{itemize}\itemsep0em
		\item The initialization problem \eqref{eq:InitializationGeneral} is equivalent to
		\begin{equation}\label{eq:InitWGeneral}
		\max_{\eps\in E_{ad}}\; [W_\circ(\eps)]_{1,1},
		\end{equation}
		where $[W_\circ(\eps)]_{i,j}:=\langle C\delta\ypmb_\circ(A_i,\eps;T),C\delta\ypmb_\circ(A_j,\eps;T)\rangle$ for $i,j\in\{1,\ldots, K\}$.
		
		\item Let $\widehat{W}_\circ^{(k)}:=\sum_{m=1}^{k}W_\circ(\eps^m)$, the fitting-step problem \eqref{eq:FittingStepGeneral} is equivalent to
		\begin{equation}\label{eq:FittingWGeneral}
		\min_{\bet\in\mathbb{R}^k}\; \langle\bet,[\widehat{W}_\circ^{(k)}]_{[1:k,1:k]}\bet\rangle-2\langle[\widehat{W}_\circ^{(k)}]_{[1:k,k+1]},\bet\rangle.
		\end{equation}
		
		\item Let $\vpmb:=[(\bet^k)^\top,\;-1]^\top$, the splitting-step problem \eqref{eq:SplittingStepGeneral} is equivalent to
		\begin{equation}\label{eq:SplittingWGeneral}
		\max_{\eps\in E_{ad}}\; \langle\vpmb,[W_\circ(\eps)]_{[1:k+1,1:k+1]}\vpmb\rangle.
		\end{equation}
	\end{itemize}
	Moreover, problems \eqref{eq:InitializationGeneral}-\eqref{eq:InitWGeneral},
	\eqref{eq:FittingStepGeneral}-\eqref{eq:FittingWGeneral}, 
	and \eqref{eq:SplittingStepGeneral}-\eqref{eq:SplittingWGeneral}
	are well posed.
\end{lemma}
\begin{proof}
    The proof is similar to the one of \cite[Lemma 5.12]{siampaper}.
	For an arbitrary $k\in\{0,\ldots,K-1\}$ let $\vpmb\in\mbbr^{k+1}$ and $A(\vpmb)=\sum_{j=1}^{k+1}\vpmb_jA_j$.
	We have $\|C\delta\ypmb_\circ(A(\vpmb),\eps;T) \|_2^2=\langle\vpmb,[W_\circ(\eps)]_{[1:k+1,1:k+1]}\vpmb\rangle$.
	%\begin{equation*}
	%\|C\delta\ypmb_\circ(\widetilde{A},\eps;T) %\|_2^2=\sum_{j=1}^{k+1}\sum_{i=1}^{k+1}\vpmb_j\vpmb_i\langle %C\delta\ypmb_\circ(A_j,\eps;T),C\delta\ypmb_\circ(A_i,\eps;T)\rangle
	%=\langle\vpmb,[W_\circ(\eps)]_{[1:k+1,1:k+1]}\vpmb\rangle.
	%\end{equation*}
	Recalling that $\delta\ypmb_\circ(A(\vpmb),\eps;T)=\sum_{j=1}^{k+1}\vpmb_j\delta\ypmb_\circ(A_j,\eps;T)$, we obtain the equivalence between \eqref{eq:InitWGeneral}, \eqref{eq:SplittingWGeneral}, and \eqref{eq:InitializationGeneral}, \eqref{eq:SplittingStepGeneral} for suitable $k$ and $\vpmb$.
	For the equivalence between \eqref{eq:FittingWGeneral} and \eqref{eq:FittingStepGeneral}, notice that for $\vpmb=[\bet^\top,-1]^\top\in\mbbr^{k+1}$ and any $W\in\mbbr^{(k+1)\times(k+1)}$ we have $\langle\vpmb,W\vpmb\rangle=\langle\bet,[W]_{[1:k,1:k]}\bet\rangle-2\langle[W]_{[1:k,k+1]},\bet\rangle+[W]_{k+1,k+1}$.
	The well-posedness of the three problems follows by standard arguments; see, e.g., \cite[Lemma 5.2]{siampaper}.
\end{proof}
The matrix representation given in Lemma \ref{lem:ALGinWGeneral} allows us to nicely describe the mathematical mechanism behind Algorithm \ref{algo:General} (see also \cite[section 5.1]{siampaper}).
Assume that at the $k$-th iteration the set $(\eps_m)_{m=1}^k$ has been computed, the submatrix $[\widehat{W}_\circ^{(k)}]_{[1:k,1:k]}$ is PD and $[\widehat{W}_\circ^{(k)}]_{[1:k+1,1:k+1]}$ has a nontrivial (one-dimensional) kernel.
Then the fitting step of Algorithm \ref{algo:General} identifies this nontrivial kernel. This can be proved by the following technical lemma (for a proof see \cite[Lemma 5.3]{siampaper}).
\begin{lemma}[kernel of some symmetric PSD matrices]\label{lem: if A pos def then also next bigger A}
	Consider a symmetric PSD matrix $\widetilde{G} = \small{\begin{bmatrix}
	G&\pmb b\\
	\pmb b^\top&c
	\end{bmatrix}}\in\mathbb{R}^{n\times n}$,
	where $G\in\mathbb{R}^{(n-1)\times(n-1)}$ is symmetric PD, 
	and $\pmb b\in\mathbb{R}^{n-1}$ and $c\in\mathbb{R}$ are such that ${\rm ker}( \widetilde{G})$ is nontrivial.	
	Then ${\rm ker}(\widetilde{G})={\rm span}\bigg\{\small{\begin{bmatrix}
	G^{-1}\pmb{b}\\-1
	\end{bmatrix}} \bigg\}$.
\end{lemma}
In our case, we have $\widetilde{G}=[\widehat{W}_\circ^{(k)}]_{[1:k+1,1:k+1]}$, $G=[\widehat{W}_\circ^{(k)}]_{[1:k,1:k]}$ and $\pmb b=[\widehat{W}_\circ^{(k)}]_{[1:k,k+1]}$.
In this notation, the solution to \eqref{eq:FittingWGeneral} is given by $\bet^k=G^{-1}\pmb{b}$.
Thus, Lemma \ref{lem: if A pos def then also next bigger A} implies that the kernel of $[\widehat{W}_\circ^{(k)}]_{[1:k+1,1:k+1]}$ is spanned by $\vpmb:=[(\bet^k)^\top,\;-1]^\top$.
Now, the splitting step attempts to compute a new control $\eps^{k+1}$ such that $[\widehat{W}_\circ(\eps^{k+1})]_{[1:k+1,1:k+1]}$ is PD on the span of $\vpmb$.
If this is successful, then $[\widehat{W}_\circ^{(k+1)}]_{[1:k+1,1:k+1]}$ is PD.
The equivalence of \eqref{eq:SplittingStepGeneral} and \eqref{eq:SplittingWGeneral} implies that $[\widehat{W}_\circ(\eps^{k+1})]_{[1:k+1,1:k+1]}$ is PD on the span of $\vpmb$ if and only if $\eps^{k+1}$ satisfies $\|C\delta\ypmb_\circ(A^{(k)}(\bet^{k}),\eps^{k+1};T)-C\delta\ypmb_\circ(A_{k+1},\eps^{k+1};T)\|_2^2>0$.
The existence of such a control depends on the controllability and observability properties of system \eqref{eq:ODElinearized}, as shown in sections \ref{sec:LinearReconstruction} and \ref{sec:BilinearReconstruction}.
We conclude this section with a remark that is useful hereafter.
\begin{remark}
	The GN matrix $\widehat{W}_\star:=\sum_{m=1}^{K}R_m'(\alstar)^\top R_m'(\alstar)\in\mbbr^{K\times K}$ can be written as $[\widehat{W}_\star]_{i,j}=\sum_{m=1}^{K}\langle C\delta\ypmb_\star(A_i,\eps^m;T),C\delta\ypmb_\star(A_j,\eps^m;T)\rangle$ for $i,j\in\{1,\ldots,K\}$, where $\delta\ypmb_\star(A_i,\eps;T)$ denotes the solution at time $T$ of
	\begin{equation*}
	\resizebox{.9\hsize}{!}{$\begin{cases}
		\dot{\delta\ypmb}_\star=\partial_{\ypmb}f(A(\alstar),\ypmb_\star,\eps)\delta\ypmb_\star+\Big(\partial_{A} f(A(\alstar),\ypmb_\star,\eps)(A_i)\Big),\quad \delta\ypmb_\star(0)=0,\\
		\dot{\ypmb}_\star=f(A(\alstar),\ypmb_\star,\eps),\quad\ypmb(0)=\ypmb^0.
		\end{cases}$}
	\end{equation*}
\end{remark}

\section{Reconstruction of drift matrix in linear systems}\label{sec:LinearReconstruction}
Consider \eqref{eq:ODE} with $f(A,\ypmb,\eps):=A\ypmb+B\eps$, where $A$ and $B$ are real matrices:
\begin{equation}\label{eq:ODElinear}
\dot{\ypmb}(t)=\Astar\ypmb(t)+B\eps(t),\; t \in(0,T], \quad\ypmb(0)=0.
\end{equation}
This is a linear system, where $B\in\mathbb{R}^{N\times M}$ is a given matrix for $N,M\in\mbbn^+$, and $\eps\in E_{ad} \subset \Ltwo$. 
%denotes a control function belonging to $E_{ad}$, a nonempty and weakly compact subset of $\Ltwo$ that contains $\eps\equiv0$ as an interior point.\footnote{This hypothesis is used in our analysis and is a reasonable assumption, since it is, for example, satisfied for standard box constraints, which are quite often used in the applications.} 
The drift matrix $\Astar\in\mathbb{R}^{N\times N}$ is unknown and assumed to lie in the space spanned by a set of linearly independent matrices $\mathcal{A}=\{A_1,\ldots,A_K\}\subset\mathbb{R}^{N\times N}$, $1\leq K\leq N^2$. We write $\Astar = \sum_{j=1}^{K}\al_{\star,j}A_j=:A(\alstar)$.
As stated in section \ref{sec:GeneralSetting}, we want to identify the unknown drift matrix $\Astar$ by using a set of control functions $(\eps^m)_{m=1}^{K}\subset E_{ad}$ in order to perform $K$ laboratory experiments and obtain the experimental data $(\vphipmb^\star_{data}(\eps^m))_{m=1}^{K}\subset\mbbr^P$, as defined in \eqref{eq:laboratorydata}.
\begin{remark}
    The hypothesis $\ypmb(0)=0$ in \eqref{eq:ODElinear} can be made without loss of generality.
    Indeed, if $\ypmb(0)=\ypmb^0 \neq 0$, one can use $\eps=0$ (case of uncontrolled system), generate the data $\vphipmb^\star_{data}(0)$, and then subtract this from all other data $(\vphipmb^\star_{data}(\eps^m))_{m=1}^{K}$ to get back (by linearity) to the case of system \eqref{eq:ODElinear} with $\ypmb(0)=0$.
    %In the case of an initial value $\ypmb(0)=\ypmb^0\in\mbbr^N\setminus\{0\}$, one can simply subtract the data from the uncontrolled system $\vphipmb^\star_{data}(0)$ from all other experimental data to get back to the case of system \eqref{eq:ODElinear}.
	%This corresponds to generating data with $\eps=0$.
	%\comment{(Gabriele comment: need to rewrite this remark.)}
\end{remark}
Using $(\eps^m)_{m=1}^{K}$ and $(\vphipmb^\star_{data}(\eps^m))_{m=1}^{K}$, the unknown vector $\alstar$ is obtained by solving \eqref{eq:identificationproblem}, in which $\ypmb(A(\al),\eps^m;T)$ now solves \eqref{eq:ODElinear}, with $A_\star$ replaced by $A(\al)$.
%\begin{equation}\label{eq:identificationlinear}
%\min_{\al\in\mbbr^K}\frac{1}{2}\sum_{m=1}^{K}\norm{\vphipmb^\star_{data}(\eps^m)-C\ypmb(A(\al),\eps^m;T)}_2^2.
%\end{equation}
Thus, we use the LGR Algorithm \ref{algo:General} to generate $(\eps^m)_{m=1}^{K}$ with the goal of making \eqref{eq:GNStepProblemCirc} uniquely solvable, that means making PD
the GN matrix $\widehat{W}_\circ$, defined in \eqref{eq:Wwidehatcirc_ij}.
In \eqref{eq:GNStepProblemCirc}, $\delta\ypmb_\circ(A(\delta\al),\eps;t)$ is now the solution to
%Since we assume to be provided with an estimate $\alcirc$ of $\alstar$, \eqref{eq:ODElinearizedCirc} becomes
\begin{equation}\label{eq:ODElinearlinearized}
\begin{cases} 
\dot{\delta\ypmb}_\circ(t)=A(\alcirc)\delta\ypmb_\circ(t)+\sum_{j=1}^{K}\delta\al_j A_j\ypmb_\circ(t),\quad t \in(0,T],\quad \delta\ypmb_\circ(0)=0, \\
\dot{\ypmb}_\circ(t)=A(\alcirc)\ypmb_\circ(t)+B\eps(t),\quad t \in(0,T],\quad \ypmb_\circ(0)=0.
\end{cases}
\end{equation}
%By denoting $\delta\ypmb_\circ(A(\delta\al),\eps;t)$ the solution of \eqref{eq:ODElinearlinearized} at time $t\in[0,T]$, the GN matrix $\widehat{W}_\circ$ is defined as in \eqref{eq:Wwidehatcirc_ij} and GR is stated in Algorithm \ref{algo:General}.

In what follows, we show that the LGR Algorithm \ref{algo:General} does produce $(\eps^m)_{m=1}^{K}$ that make $\widehat{W}_\circ$ PD under appropriate assumptions on observability and controllability of the considered linear system.
Let us recall these properties for an input/output system $(A,B,C)$ of the form \eqref{eq:IOSyst}
with $A\in\mathbb{R}^{N \times N}$, $B\in\mathbb{R}^{N \times M}$, $C\in\mathbb{R}^{P \times N}$; see, e.g., \cite[Theorem 3, Theorem 23]{Sontag1998}.
\begin{definitionandlemma}[observable input-output linear systems]\label{def:ObservabilityLinear}
	The linear system \eqref{eq:IOSyst} is said to be observable if the initial state $\ypmb(0)=\ypmb^0$ can be uniquely determined from input/output measurements.
	Equivalently, \eqref{eq:IOSyst} is observable if and only if the observability matrix $\mathcal{O}_N(C,A) := \begin{bmatrix}
	C& CA& \cdots&CA^{N-1}
	\end{bmatrix}^\top$ has full rank.
\end{definitionandlemma}
\begin{definitionandlemma}[controllable input-output linear systems]
	The linear system \eqref{eq:IOSyst} is said to be controllable if for any final state $\ypmb^f$ there exists an input sequence that transfers $\ypmb^0$ to $\ypmb^f$. 
	Equivalently, \eqref{eq:IOSyst} is controllable if and only if the controllability matrix $\label{eq:controllabilitymatrix}
	\mathcal{C}_N(A,B) := \begin{bmatrix}
	B& AB& \cdots&A^{N-1}B
	\end{bmatrix}$ has full rank.
\end{definitionandlemma}
%%%%% SUPPL
%In Section \ref{subsec:convanalysisfullyobscont}, we analyze Algorithm~\ref{algo:General} in the case of fully observable and controllable systems (namely, $\ranktxt(\obsmatzero)=\ranktxt(\contmatzero)=N$).
%However, similar to \cite[Section 5.3]{siampaper}, one can also formulate the following results for non-fully observable and controllable systems, if appropriate matrices $A_1,\ldots,A_K$ are chosen.
%For further details, we refer the reader to the supplementary material \cite{supp_material}.
%%%%% SUPPL

Notice that the analysis that we are going to present is also valid in the case of the reconstruction of a control matrix considered in \cite[Section 5]{siampaper}, i.e. $f(A,\ypmb,\eps)=M\ypmb+A\eps$, and is therefore an extension of the results obtained in \cite{siampaper}.

\subsection{Analysis for linear systems}\label{subsec:convanalysisfullyobscont}
We define $\mathcal{O}_N^\circ:=\obsmatzero$ and $\mathcal{C}_N^\circ:=\contmatzero$
and assume that the system $(A(\alcirc),B,C)$ is observable and controllable, namely $\mathcal{R}:=\ranktxt(\obsmatzeroSHORT)\cdot\ranktxt(\contmatzeroSHORT)=N^2$.
%For brevity we denote $\mathcal{O}_N^\circ:=\obsmatzero$ and $\mathcal{C}_N^\circ:=\contmatzero$.
In what follows, we show that this is a sufficient condition for $\widehat{W}_\circ$ to be PD with the controls generated by Algorithm \ref{algo:General}. First, we need the following result \cite[Ch.~3, Theorem 2.11]{linearsystems}.
\begin{lemma}[controllability of time-invariant systems]\label{lem:reachability}
	Consider the system $\dot{\xpmb}=A\xpmb+B\eps$ with $\xpmb(0)=0$ and its solution $\xpmb(\eps,t):=\int_{0}^{t}e^{(t-s)A(\alcirc)}B\eps(s)ds$. For any finite time $t_0>0$, there exists a control $\eps$ that transfers the state to $\wpmb$ in time $t_0$, i.e. $\xpmb(\eps,t_0)=\wpmb$, if and only if 
	%$\wpmb$ is in the image of the controllability matrix $\mathcal{C}_N(A,B)$, i.e. 
	$\wpmb\in\imtxt\Big(\mathcal{C}_N(A,B)\Big)$.
	Furthermore, an appropriate $\eps$ that will accomplish this transfer in time $t_0$ is given by $\eps(t)=B^\top e^{(t_0-t)A^\top}\nupmb$, for $t\in[0,t_0]$ and $\nupmb$ such that $\mathcal{W}_c(0,t_0)\nupmb=\wpmb$, where $\mathcal{W}_c(0,T):=\int_{0}^{T}e^{\tau A}BB^\top e^{\tau A^\top}d\tau$.% is the Gramian of \eqref{eq:IOSyst}. 
\end{lemma}
Now, we prove the following lemma regarding the initialization problem \eqref{eq:InitializationGeneral} and the splitting step problem \eqref{eq:SplittingStepGeneral}. Notice that the proof of this result is inspired by classical Kalman controllability theory; see, e.g., \cite{coron2007control}.
\begin{lemma}[LGR initialization and splitting steps (linear systems)]\label{lem:positivitymaxproblems}
	Assume that the matrices $A(\alcirc)\in\mathbb{R}^{N \times N}$, $B\in\mathbb{R}^{N \times M}$ and $C\in\mathbb{R}^{P \times N}$ are such that\linebreak $\ranktxt(\mathcal{O}_N^\circ)=\ranktxt(\mathcal{C}_N^\circ)=N$, and let $\widetilde{A}\in\mbbr^{N\times N}\setminus\{0\}$ be arbitrary. 
	Then any solution $\widetilde{\eps}$ of the problem $\max_{\eps\in E_{ad}}\|C\delta\ypmb_\circ\klam{\widetilde{A},\eps;T}\|_2^2$ satisfies $$\|C\delta\ypmb_\circ\klam{\widetilde{A},\widetilde{\eps};T}\|_2^2>0,$$ where $\dot{\delta\ypmb}_{\circ}=A(\alcirc)\delta\ypmb_\circ+\widetilde{A}\ypmb^\circ$, with $\delta\ypmb_\circ(0)=0$, and $\dot{\ypmb}_\circ=A(\alcirc)\ypmb_\circ+B\eps$ with $\ypmb_\circ(0)=0$
\end{lemma}
\begin{proof}
    To prove the result, it is sufficient to construct an $\widetilde{\eps}\in E_{ad}$ such that $C\delta\ypmb_\circ\klam{\widetilde{A},\widetilde{\eps};T}\neq0$.
	Since $\widetilde{A}\neq0$, there exists $\wpmb\in\mbbr^N\setminus\{0\}$ such that $\widetilde{A}\wpmb\neq0$. Since $(A(\alcirc),B,C)$ is observable, there exists $\tilde{t}>0$ such that $Ce^{\tilde{t}A(\alcirc)}\widetilde{A}\wpmb\neq0$. 
	The map $f:\mbbr\rightarrow\mbbr^{P},t\mapsto Ce^{tA(\alcirc)}\widetilde{A}\wpmb$ is analytic with derivatives $f^{(i)}(t)=CA(\alcirc)^ie^{tA(\alcirc)}\widetilde{A}\wpmb$. 
	Since $\mathcal{O}_N^\circ$ has full rank and $e^{\tilde{t}A(\alcirc)}\widetilde{A}\wpmb\neq0$, there exists $i\in\{0,\ldots,N\}$ such that $f^{(i)}(\tilde{t})=CA(\alcirc)^ie^{\tilde{t}A(\alcirc)}\widetilde{A}\wpmb\neq0$. 
	Hence, $f$ is nonconstant, and there exists $t_0\in(0,T)$ with $Ce^{t_0A(\alcirc)}\widetilde{A}\wpmb\neq0$.
	
	Now, we use that $\ypmb_\circ(\eps,s):=\int_{0}^{s}e^{(s-\tau)A(\alcirc)}B\eps(\tau)d\tau$ is the solution at time $s$ of $\dot{\ypmb}_\circ=A(\alcirc)\ypmb_\circ+B\eps$, with $\ypmb_\circ(0)=0$. 
	Since $\mathcal{C}_N^\circ$ has full rank, we have $\wpmb\in\imtxt\big(\mathcal{C}_N^\circ\big)$. 
	Thus, Lemma \ref{lem:reachability} guarantees that $\widehat{\eps}(t)=B^\top e^{(t_0-t)A(\alcirc)^\top}\nupmb$, for $t\in[0,t_0]$ and some $\nupmb\in\mbbr^N$, satisfies $\ypmb_\circ(\widehat{\eps},t_0)=\wpmb$. Clearly, $\widehat{\eps}$ is analytic in $[0,t_0]$ and thereby the same holds for $\ypmb_\circ(\widehat{\eps},s)$.
	Note that, since $\eps\equiv0$ is an interior point of $E_{ad}$, there exists $\lambda>0$ such that $\lambda\widehat{\eps}\in E_{ad}$ with $Ce^{t_0A(\alcirc)}\widetilde{A}\ypmb_\circ(\lambda\widehat{\eps},t_0)=\lambda Ce^{t_0A(\alcirc)}\widetilde{A}\ypmb_\circ(\widehat{\eps},t_0)\neq0$. Hence, we can assume without loss of generality that $\widehat{\eps}\in E_{ad}$.
	
	In conclusion, we obtain that the map
	\begin{equation*}
	\gpmb:\mbbr\rightarrow\mbbr^p,s\mapsto Ce^{(T-s)A(\alcirc)}\widetilde{A}\int_{0}^{s}e^{(s-\tau)A(\alcirc)}B\widehat{\eps}(\tau)d\tau
	\end{equation*}
	is analytic in $(0,t_0)$ with $\gpmb(t_0)\neq0$. 
	Thus, $\gpmb$ is nonzero in an open subinterval of $(0,t_0)$.
	Hence, there exists $t_1\in(0,t_0)$ such that $\int_{0}^{t_1}\gpmb(s)ds\neq0$. By choosing 
	\begin{equation*}
	\widetilde{\eps}(s) :=\begin{cases}
	0,& 0\leq s< T-t_1,\\
	\widehat{\eps}(s-t_1),& T-t_1\leq s\leq T,
	\end{cases}
	\end{equation*}
	and using that $C\delta\ypmb_\circ\klam{\widetilde{A},\widetilde{\eps};T}=\int_{0}^{T}Ce^{(T-s)A(\alcirc)}\widetilde{A}\int_{0}^{s}e^{(s-\tau)A(\alcirc)}B\widetilde{\eps}(\tau)d\tau ds$, we obtain
	
	\begin{align*}
	C\delta\ypmb_\circ\klam{\widetilde{A},\widetilde{\eps};T}
	&=\int_{T-t_1}^{T}Ce^{(T-s)A(\alcirc)}\widetilde{A}\int_{T-t_1}^{s}e^{(s-\tau)A(\alcirc)}B\widetilde{\eps}(\tau-t_1)d\tau ds\\
	&=\int_{0}^{t_1}Ce^{(t_1-s)A(\alcirc)}\widetilde{A}\int_{0}^{s}e^{(s-\tau)A(\alcirc)}B\widehat{\eps}(\tau)d\tau ds=\int_{0}^{t_1}\gpmb(s)ds\neq0.
	\end{align*}
\end{proof}
Lemma \ref{lem:positivitymaxproblems} can be applied to both \eqref{eq:InitializationGeneral} and \eqref{eq:SplittingStepGeneral}, choosing $\widetilde{A}=A_1$ and $\widetilde{A}=\big(A^{(k)}(\bet^k)-A_{k+1}\big)$, respectively.
Now, we can prove our first main convergence result.
\begin{theorem}[positive definiteness of the GN matrix $\widehat{W}_\circ$ (linear systems)]\label{thm:convergence_fully_obscont}
	Assume that $A(\alcirc)\in\mathbb{R}^{N \times N}$, $B\in\mathbb{R}^{N \times M}$ and $C\in\mathbb{R}^{P \times N}$ are such that $\ranktxt(\mathcal{O}_N^\circ)=\ranktxt(\mathcal{C}_N^\circ)=N$. 
	For $K\leq N^2$, let $\mathcal{A}=\{A_1,\ldots,A_K\}\subset~\mbbr^{N\times N}$ be a set of linearly independent matrices such that $A(\alcirc)\in\spantxt(\mathcal{A})$, and let $\{\eps^1,\ldots,\eps^{K} \}\subset E_{ad}$ be generated by Algorithm \ref{algo:General}. 
	Then the GN matrix $\widehat{W}_\circ$, defined in \eqref{eq:Wwidehatcirc_ij}, is PD.
\end{theorem}
\begin{proof} 
	We proceed by induction.
	Lemma \ref{lem:positivitymaxproblems} guarantees that there exists an $\eps^1$ such that $[W_\circ(\eps^1)]_{1,1}=\norm{C\delta\ypmb_\circ(A_1,\eps;T)}_2^2>0$.
	Now, we assume that $[\widehat{W}_\circ^{(k)}]_{[1:k,1:k]}=\sum_{m=1}^{k}[W_\circ(\eps^m)]_{[1:k,1:k]}$ is PD. 
	By construction, $[\widehat{W}_\circ^{(k+1)}]_{[1:k+1,1:k+1]}$ is PSD.
	Thus, if $[\widehat{W}_\circ^{(k)}]_{[1:k+1,1:k+1]}$ is PD, then 
	\begin{equation*}
	[\widehat{W}_\circ^{(k+1)}]_{[1:k+1,1:k+1]} = [\widehat{W}_\circ^{(k)}]_{[1:k+1,1:k+1]} + [W_\circ(\eps^{k+1})]_{[1:k+1,1:k+1]}
	\end{equation*}
	is PD as well, since $[W_\circ(\eps^k)]_{[1:k+1,1:k+1]}$ is PSD.
	Assume now that the submatrix $[\widehat{W}_\circ^{(k)}]_{[1:k+1,1:k+1]}$ has a nontrivial kernel. 	
	Since $[\widehat{W}_\circ^{(k)}]_{[1:k,1:k]}$ is PD (induction hypothesis), problem \eqref{eq:FittingStepGeneral} is uniquely solvable with solution $\bet^k$. 
	Then, by Lemma \ref{lem: if A pos def then also next bigger A} the (one-dimensional) kernel of $[\widehat{W}_\circ^{(k)}]_{[1:k+1,1:k+1]}$ is the span of the vector $\vpmb =[(\bet^k)^\top,\;-1]^\top$.
	Using Lemma \ref{lem:positivitymaxproblems} we obtain that the solution $\eps^{k+1}$ to the splitting step problem satisfies
	\begin{equation*}
		\langle\vpmb,[W_\circ(\eps^{k+1})]_{[1:k+1,1:k+1]}\vpmb\rangle=\norm{C\delta\ypmb_\circ(A^{(k)}(\bet^k)-A_{k+1},\eps;T)}_2^2>0.
	\end{equation*}
	Thus, $[W(\eps^{k+1})]_{[1:k+1,1:k+1]}$ is PD on the span of $\vpmb$, and
	$[\widehat{W}_\circ^{(k+1)}]_{[1:k+1,1:k+1]}$
	%=[\widehat{W}_\circ^{(k)}]_{[1:k+1,1:k+1]}+[W_\circ(\eps^{k+1})]_{[1:k+1,1:k+1]}$
	is PD. 
\end{proof}
%%%%% SUPPL
%Notice that Theorem \ref{thm:convergence_fully_obscont} does not require any assumption on the matrices $A_1,\dots,A_K$. These can be arbitrarily chosen with the only constraint to be linearly independent.
%Also the ordering of these matrices does not affect the result of Theorem \ref{thm:convergence_fully_obscont}.
%This is, however, different for non-fully observable and controllable systems, i.e. for $\mathcal{R}<N^2$ (see the supplementary material \cite{supp_material}).
%%%%% SUPPL

Now that we proved that Algorithm~\ref{algo:General} makes $\Wcirc$ PD, the obvious question is whether this is sufficient for the convergence of GN, as described in Lemma \ref{lem:gaussnewton}. We answer this question in Section \ref{subsec:PosDefGaussNewtonLinear}.

\subsection{Positive definiteness of the GN matrix}\label{subsec:PosDefGaussNewtonLinear}
To guarantee convergence of GN, we need to show that $\widehat{W}(\al):=\sum_{m=1}^{K}R_m'(\al)^\top R_m'(\al)$ (defined in section \ref{sec:GeneralSetting})
remains PD in a neighborhood of $\al_\star$.
Indeed, in Section \ref{subsec:convanalysisfullyobscont}, we proved that the control functions generated by Algorithm \ref{algo:General} make the GN matrix $\widehat{W}_\circ = \widehat{W}(\al_\circ)$ PD. 
Thus, it is sufficient to prove that $\widehat{W}(\al)$ 
remains PD in a neighborhood of $\al_\circ$ containing $\al_\star$.
To do so, let us rewrite $\widehat{W}(\al)$ as
\begin{align}
[\widehat{W}(\al)]_{i,j}&:=\sum_{m=1}^{K}\langle\gam_i(\al,\eps^m),\gam_j(\al,\eps^m)\rangle,\quad i,j\in\{1,\ldots,K\},\label{eq:Wijstar}\\
\gam_j(\al,\eps^m)&:=\int_{0}^{T}Ce^{(T-s)A(\al)}A_j\ypmb(A(\al),\eps^m;s)ds,\quad j\in\{1,\ldots,K\},\label{eq:Gammajstar}
\end{align}
and recall the next lemma, which follows from the Bauer-Fike theorem \cite{BAUER1960}.
\begin{lemma}[rank stability]\label{lem:rankdisturbedmatrix}
	Consider two natural numbers $N_D$ and $M_D$ with $N_D\geq M_D$, and an arbitrary matrix $D\in\mbbr^{N_D\times M_D}$ with rank $\mathcal{R_D}$ and (positive) singular values $\sigma_1,\ldots,\sigma_{\mathcal{R_D}}$ in descending order. Then it holds that
	\begin{equation*}
	\min_{\widehat{D}\in\mbbr^{N_D\times M_D}}\{\|\widehat{D}\|_2\;\mid\; \ranktxt(D+\widehat{D})<\mathcal{R_D} \}=\sigma_{\mathcal{R_D}}.
	\end{equation*}
\end{lemma}
Using this lemma, we can prove the following approximation result.
\begin{lemma}[positive definiteness of $\widehat{W}(\al)$ (linear systems)]\label{lem:ApproximationWhatalpha}
	Let $\widehat{W}_\circ$ defined in \eqref{eq:Wwidehatcirc_ij} be PD and let $\sigma^\circ_K>0$ be its smallest singular value. Then, there exists $\delta:=\delta(\sigma^\circ_K)>0$ such that $\widehat{W}(\al)$ (in \eqref{eq:Wijstar}) is PD for any $\al\in\mbbr^K$ with $\|\al-\alcirc\|_2<\delta$.
\end{lemma}
\begin{proof}
	Our first goal is to show that $\widehat{W}(\al)$ is continuous in $\al$. From \eqref{eq:Wijstar} and \eqref{eq:Gammajstar} we know that $\widehat{W}(\al)$ is the sum over products of $\int_{0}^{T}Ce^{(T-s)A(\al)}A_j\ypmb(A(\al),\eps^m;s)ds$, where $\ypmb(A(\al),\eps^m;s)=\int_{0}^{s}e^{(s-\tau)A(\al)}B\eps^m(\tau)d\tau$.
	Now, recall that $A(\al)=\sum_{j=1}^{K}\al_jA_j$, meaning that $A(\al)$ is continuous in $\al$. 
	Since the exponential map $\mbbr^{N}\rightarrow\mbbr^{N\times N}, \al\mapsto e^{sA(\al)}$ and the integral map $\mbbr^{N\times N}\rightarrow\mbbr^N, X\mapsto \int_{0}^{s}XB\eps(\tau)d\tau$ are continuous, we obtain that $\ypmb(A(\al),\eps^m;s)$ is continuous in $\al$. 
	Since products of continuous functions are continuous, we obtain that $\widehat{W}(\al)$ is continuous in $\al$.
	
	By assumption, $\widehat{W}_\circ$ is PD, and therefore $\sigma^\circ_{K}>0$. Since $\widehat{W}(\al)$ is continuous in $\al$, we obtain that there exists a $\delta:=\delta(\sigma^\circ_{K})>0$ such that for any $\al$ with $\norm{\al-\alcirc}_2<\delta(\sigma^\circ_{K})$ it holds that $\norm{\widehat{W}(\al)-\widehat{W}\klam{\alcirc}}_2<\sigma^\circ_{K}$. 
	Now, let $\widehat{\al}$ be such that $\norm{\widehat{\al}-\alcirc}_2<\delta(\sigma^\circ_{K})$ and hence $\norm{\widehat{W}(\widehat{\al})-\widehat{W}(\alcirc)}_2<\sigma^\circ_{K}$.
	Setting $D=\widehat{W}(\alcirc)$ and $\widehat{D}=\widehat{W}(\widehat{\al})-\widehat{W}(\alcirc)$, Lemma \ref{lem:rankdisturbedmatrix} implies that $K=\ranktxt(\widehat{W}(\alcirc))\leq\ranktxt(\widehat{W}(\widehat{\al}))$. 
	Because of~\eqref{eq:Wijstar}, $\widehat{W}(\widehat{\al})\in\mbbr^{K\times K}$ meaning that $\ranktxt(\widehat{W}(\widehat{\al}))=K$.
	Since $\widehat{W}(\al)$ is PSD by construction, $\ranktxt(\widehat{W}(\widehat{\al}))=K$ implies that $\widehat{W}(\widehat{\al})$ is PD.
\end{proof}

Lemma \ref{lem:ApproximationWhatalpha} implies that the positive definiteness of $\widehat{W}(\al)$ is locally preserved near $\al_\circ$.
Now, we can prove our main convergence result.

\begin{theorem}[convergence of GN (linear systems)]\label{thm:ConvergenceTrueSystem}
	Let $\alcirc\in\mbbr^K$ be such that the matrices $A(\alcirc)\in\mathbb{R}^{N \times N}$, $B\in\mathbb{R}^{N \times M}$ and $C\in\mathbb{R}^{P \times N}$ satisfy $\ranktxt(\mathcal{O}_N^\circ)\cdot\ranktxt(\mathcal{C}_N^\circ)=N^2$.
	Let $(\eps^m)_{m=1}^K\subset E_{ad}$ be a set of controls generated by Algorithm~\ref{algo:General}.
	Finally, let $\widehat{\sigma}_K$ be the $K$-th (smallest) singular value of $\widehat{W}_\circ$ defined in \eqref{eq:Wwidehatcirc_ij}. 
	Then there exists $\delta=\delta(\widehat{\sigma}_K)>0$ such that if $\alstar\in\mbbr^K$ satisfies $\|\alstar-\alcirc\|<\delta$, then GN method for the problem
	\begin{equation}\label{eq:identificationlinear}
	\min_{\al\in\mbbr^K}\frac{1}{2}\sum_{m=1}^K\norm{C\ypmb(A(\alstar),\eps^m;T)-C\ypmb\klam{A(\al),\eps^m;T}}_2^2,
	\end{equation}
	initialized with $\alcirc$, converges to $\al_j=\al_{\star,j}$, $j=1,\ldots,K$.
\end{theorem}
\begin{proof}
	Theorem \ref{thm:convergence_fully_obscont} guarantees that $\widehat{W}_\circ$ is PD and hence $\widehat{\sigma}_K>0$. 
	Thus, by Lemma \ref{lem:ApproximationWhatalpha} there exists $\delta=\delta(\widehat{\sigma}_K)>0$ such that, for $\al\in\mbbr^K$ with $\|\al-\alcirc\|_2<\delta$, the matrix $\widehat{W}(\al)$ is also PD.
	Moreover, we know from section \ref{sec:GeneralSetting} that $\widehat{W}(\al_c)$ is the GN matrix for the iterate $\al_c\in\mbbr^K$ of GN for \eqref{eq:identificationproblem}.
	Analogously to the proof of Lemma \ref{lem:ApproximationWhatalpha}, one can also show that the functions $R_m(\al)$, defined in \eqref{eq:rm}, are Lipschitz continuously differentiable in $\al$ for all $m\in\{1,\ldots,K\}$.
	Hence, if $\|\alstar-\alcirc\|<\delta$, then the result follows by Lemma \ref{lem:gaussnewton}.
\end{proof}

\subsection{Local uniqueness of solutions}\label{subsec:LocUniqSolLinear}
Theorem \ref{thm:ConvergenceTrueSystem} says that GN converges to $\alstar$ if an appropriate initialization vector $\alcirc$ is used.
However, %for 
in the linear case corresponding to  %the linearity of % case of system 
\eqref{eq:ODElinear} %, we can even say more about 
we can specify the local properties of problem~\eqref{eq:identificationproblem} around the solution $\alstar$.
To this end, we start by rewriting the cost function in a matrix form.
\begin{lemma}[online identification problem in matrix form (linear systems)]\label{lem:identificationinmatrixform}
	Problem \eqref{eq:identificationproblem} is equivalent to
	\begin{equation}\label{eq:identificationproblemW}
	\min_{\al\in\mbbr^K}
	\frac{1}{2}\langle\alstar-\al,\widetilde{W}\klam{\alstar,\al}(\alstar-\al)\rangle,
	\end{equation}
	where $\widetilde{W}\klam{\alstar,\al}\in\mbbr^{K\times K}$ is defined as\footnote{Notice that the notations \eqref{eq:Wijstar} and \eqref{eq:Wwidehatstar} are related in the sense that $\widetilde{W}(\al,\al)=\widehat{W}(\al)$.}
	\begin{equation}\label{eq:Wwidehatstar}
	\widetilde{W}\klam{\alstar,\al} := \sum_{m=1}^{K}W\klam{\alstar,\al,\eps^m},
	\end{equation}
	with $W\klam{\alstar,\al,\eps^m}\in\mbbr^{K\times K}$ given by
	\begin{align}
	[W\klam{\alstar,\al,\eps^m}]_{i,j}&:=\langle\gam_i\klam{\alstar,\al,\eps^m},\gam_j\klam{\alstar,\al,\eps^m}\rangle,\quad i,j\in\{1,\ldots,K\},\label{eq:Walijstar}\\
	\gam_j\klam{\alstar,\al,\eps^m}&:=\int_{0}^{T}Ce^{(T-s)A(\alstar)}A_j\ypmb\klam{A(\al),\eps^m;s}ds,\quad j\in\{1,\ldots,K\}.\label{eq:Gammaaljstar}
	\end{align}
\end{lemma}
\begin{proof}
	Let $J(\al):=\frac{1}{2}\sum_{m=1}^{K}\norm{C\ypmb\klam{\Astar,\eps^m;T}-C\ypmb\klam{A(\al),\eps^m;T}}_2^2$. For $t\in[0,T]$ and $m\in\{1,\ldots,K \}$ define $\Delta\ypmb_m(t):=\ypmb\klam{\Astar,\eps^m;t}-\ypmb\klam{A(\al),\eps^m;t}$.
	Then we have
	\begin{align*}
	\dot{\Delta\ypmb}_m(t)&=A(\alstar)\ypmb\klam{\Astar,\eps^m;t}+B\eps^m(t)-A(\al)\ypmb\klam{A(\al),\eps^m;t}-B\eps^m(t)\\
	&=A(\alstar)\Delta\ypmb_m(t) + A(\alstar-\al)\ypmb\klam{A(\al),\eps^m;t},
	\end{align*}
	whose solution at time $T$ is given by
	\begin{equation*}
	\Delta\ypmb_m(T)=\int_{0}^{T}e^{(T-s)A(\alstar)}\Big[A(\alstar-\al)\ypmb\klam{A(\al),\eps^m;s}\Big]ds.
	\end{equation*}
	Thus, recalling $A(\al)=\sum_{j=1}^{K}\al_jA_j$, the function $J(\al)$ can be written as
	\begin{align*}
	J(\al)&=\frac{1}{2}\sum_{m=1}^{K}\norm{\int_{0}^{T}Ce^{(T-s)A(\alstar)}\Big(\sum_{j=1}^{K}(\al_{\star,j}-\al_j)A_k \Big)\ypmb\klam{A(\al),\eps^m;s}ds}_2^2\\
	&\overset{\eqref{eq:Gammaaljstar}}{=}\frac{1}{2}\sum_{m=1}^{K}\sum_{i=1}^{K}\sum_{j=1}^{K}(\al_{\star,i}-\al_i)(\al_{\star,j}-\al_j)\langle\gam_i\klam{\alstar,\al,\eps^m},\gam_j\klam{\alstar,\al,\eps^m}\rangle\\
	&\overset{\eqref{eq:Walijstar}}{=}\frac{1}{2}\langle\alstar-\al,\sum_{m=1}^{K}W\klam{\alstar,\al,\eps^m}(\alstar-\al)\rangle=\frac{1}{2}\langle\alstar-\al,\widetilde{W}\klam{\alstar,\al}(\alstar-\al)\rangle.
	\end{align*}
\end{proof}
Now, the set of global solutions to problem \eqref{eq:identificationproblemW} is given by $\mathcal{S}_{global}:= \Big\{\al\in\mbbr^K : (\alstar-\al)\in\kertxt\; \widetilde{W}(\alstar,\al) \Big\}$.
Since $\widetilde{W}(\alstar,\al)$ is symmetric PSD, \eqref{eq:identificationproblemW} is locally uniquely solvable if and only if $\widetilde{W}(\alstar,\al)$ is PD for $\al$ close to $\alstar$.
Now, assume that the system is fully observable and controllable, meaning that $\mathcal{R}=N^2$.
Theorem \ref{thm:ConvergenceTrueSystem} guarantees that Algorithm \ref{algo:General} computes $(\eps_m)_{m=1}^{N^2}$ such that $\widehat{W}(\alstar)=\widetilde{W}(\alstar,\alstar)$ is PD, if $\alstar$ is close enough to the estimate $\alcirc$.
Similar to the proof of Lemma \ref{lem:ApproximationWhatalpha}, one can prove that $\widetilde{W}(\alstar,\al)$ is continuous in $\al$.
Hence, we obtain that if the matrix $\widetilde{W}(\alstar,\alstar)$ is PD, then the same is true for $\widetilde{W}(\alstar,\al)$, when $\al$ is close to $\alstar$, which implies that \eqref{eq:identificationproblemW} is locally uniquely solvable with $\al=\alstar$.

\section{Bilinear reconstruction problems}\label{sec:BilinearReconstruction}
In this section, we extend the results of section~\ref{sec:LinearReconstruction} to the case of skew-symmetric bilinear systems. We consider \eqref{eq:ODE} with a right-hand side $f(A,\ypmb,\epsilon)=(A+\epsilon B)\ypmb$, that is
\begin{equation}\label{eq:ODEbilinear}
\dot{\ypmb}(t)=(\Astar+\epsilon(t) B)\ypmb(t),\; t \in(0,T], \quad\ypmb(0)=\ypmb^0,
\end{equation}
where $B\in\mathfrak{so}(N)$ is a given skew-symmetric matrix for $N\in\mbbn^+$, the initial state is $\ypmb^0\in\mathbb{R}^N$, and \comment{$\epsilon\in E_{ad}\subset L^2(0,T;\mbbr)$ (with $M=1$; see section \ref{sec:notation}) is a control function.}
%, a nonempty, closed, convex and bounded subset of $L^2(0,T;\mbbr)$ that contains $\epsilon\equiv0$ as an interior point. 
The matrix $\Astar\in\mathfrak{so}(N)$ is unknown and assumed to lie in the space spanned by a set of linearly independent matrices $\mathcal{A}=\{A_1,\ldots,A_K\}\subset\mathbb{R}^{N\times N}$, $1\leq K\leq N^2$, and we write $\Astar = \sum_{j=1}^{K}\al_{\star,j}A_j=:A(\alstar)$.
%Notice that, 
\comment{Since} the matrices $\Astar$ and $B$ are skew-symmetric, system \eqref{eq:ODEbilinear} is norm preserving, i.e. $\|\ypmb(t)\|_2=\|\ypmb^0\|_2$ for all $t\in[0,T]$.\footnote{To see this, we observe that $\frac{1}{2}\frac{d}{dt}\|\ypmb(t)\|^2_2=\langle\ypmb(t),\dot{\ypmb}(t)\rangle=\langle\ypmb(t),(\Astar+\epsilon(t) B)\ypmb(t)\rangle=0$.}
%(see, e.g., \cite[Example 3.2]{LibroQuantum}).

To identify the true matrix $\Astar$, one can consider a set of control functions $(\epsilon^m)_{m=1}^{K}\subset E_{ad}$ and use it experimentally to obtain the data $(\vphipmb^\star_{data}(\epsilon^m))_{m=1}^K\subset\mbbr^P$, as defined in \eqref{eq:laboratorydata}.
The unknown vector $\alstar$ is then obtained by solving the problem
\begin{equation}\label{eq:IdentificationProblemBilinear}
\min_{\al\in\mbbr^K}\frac{1}{2}\sum_{m=1}^{K}\norm{\vphipmb^\star_{data}(\epsilon^m)-C\ypmb(A(\al),\epsilon^m;T)}_2^2.
\end{equation}
We assume to be provided with a known estimate $\alcirc$ of $\alstar$.
For this estimate, we can derive the linearized equation
\begin{equation}\label{eq:ODElinearizedBilinear}
\resizebox{.91\hsize}{!}{$\begin{cases} 
\dot{\delta\ypmb}_{\circ}(t)=(A_\circ+\epsilon(t)B)\delta\ypmb_\circ(t)+\sum_{j=1}^{K}\delta\al_j A_j\ypmb_\circ(t),\quad t \in(0,T],\quad \delta\ypmb_\circ(0)=0, \\
\dot{\ypmb}_\circ(t)=(A_\circ+\epsilon(t)B)\ypmb_\circ(t),\quad t \in(0,T],\quad \ypmb_\circ(0)=\ypmb^0,
\end{cases}$}
\end{equation}
where $A_\circ:=A(\alcirc)$.
Denoting by $\delta\ypmb_\circ(A(\delta\al),\epsilon;t)$ the solution of \eqref{eq:ODElinearizedBilinear} at time $t\in[0,T]$, the GN matrix $\widehat{W}_\circ$ is defined as in \eqref{eq:Wwidehatcirc_ij}, and LGR is detailed in Algorithm \ref{algo:General}.

Let us recall the following definition and result from \cite[Corollary 4.11]{LibroQuantum}. % regarding the controllability of bilinear systems, applied to the real case of system \eqref{eq:ODEbilinear}.
\begin{definitionandlemma}[Controllability of skew-symmetric bilinear systems]\label{deflem:ContrBilinear}
	Consider a system of the form 
	\begin{equation}\label{eq:BilinearSystem}
	\dot{\ypmb}(t)=(A_\circ+\epsilon(t)B)\ypmb(t),\quad \ypmb(0)=\ypmb^0,
	\end{equation}
	where $A_\circ,B\in\mathfrak{so}(N)$. 
	System \eqref{eq:BilinearSystem} is said to be controllable if for any final state $\ypmb^f$ that lies on the sphere of radius $\|\ypmb^0\|_2$ there exists a control $\epsilon(t)$ that transfers $\ypmb^0$ to $\ypmb^f$.
	Furthermore, if the Lie algebra $L=Lie\{A_\circ,B\}\subset\mathfrak{so}(N)$, generated by the matrices $A_\circ$ and $B$, has dimension $\frac{N(N-1)}{2}$, then there exists a constant $\tilde{t}\geq0$ such that for any $T\geq \tilde{t}$ controllability of \eqref{eq:BilinearSystem} holds.
\end{definitionandlemma}

As in section \ref{sec:LinearReconstruction}, we also need to make some assumptions on the observability of the linearized equation in \eqref{eq:ODElinearizedBilinear}.
However, recalling the proof of Lemma \ref{lem:positivitymaxproblems}, these assumptions are only required to prove the existence of a control function that guarantees a positive cost function value in the splitting step.
If we assume this function to be constant, at least on a subinterval of $[0,T]$, then we get a system of the form
\begin{equation}\label{eq:ODEBilinearConstantControl}
\dot{\delta\ypmb}_\circ(t)=(A_\circ+cB)\delta\ypmb_\circ(t)+A(\delta\al)\ypmb_\circ(t),
\end{equation}
for a scalar $c\in\mbbr$.
In this case, system \eqref{eq:ODEBilinearConstantControl} is again a linear system, for which observability is defined in Definition \ref{def:ObservabilityLinear}.
Hence, the observability matrix is $\mathcal{O}_N(C,A_\circ+cB)$.
%given by
%\begin{equation}\label{eq:ObservabilityMatrixBilinear}
%\mathcal{O}_N(C,A_\circ+cB) = \begin{bmatrix}
%C& C(A_\circ+cB)& \cdots& C(A_\circ+cB)^{N-1}
%\end{bmatrix}^\top.
%\end{equation}
Let us state our assumptions on controllability and observability of \eqref{eq:BilinearSystem} and \eqref{eq:ODEBilinearConstantControl}.
\begin{assumption}\label{assum:Bilinear}
	Let the matrices $A_\circ$, $B$ and $C$ be such that the following conditions are satisfied.
	\begin{enumerate}\itemsep0em
		\item The Lie algebra $L=Lie\{A_\circ,B\}\subset\mathfrak{so}(N)$, generated by the matrices $A_\circ$ and $B$, has dimension $\frac{N(N-1)}{2}$.
		\item The final time $T>0$ is sufficiently large, such that the controllability result from Lemma \ref{deflem:ContrBilinear} holds.
		\item There exists $c\in\mbbr$ such that system \eqref{eq:ODEBilinearConstantControl} is observable, i.e. the observability matrix $\mathcal{O}_N(C,A_\circ+cB)$ has full rank.
	\end{enumerate}
	In addition, let the set of admissible controls $E_{ad}\subset L^2(0,T;\mbbr)$ be chosen such that the controllability result from Lemma \ref{deflem:ContrBilinear} holds, and such that $\epsilon\equiv c$ is an interior point of $E_{ad}$ for the constant $c\in\mbbr$ mentioned above.
\end{assumption}
%\begin{remark}
%	Notice that the assumptions for this bilinear setting are more restrictive and than in the linear case of section \ref{sec:LinearReconstruction}.
%	We also do not consider the case of not fully controllable and observable systems.
%	This is mainly because, we require a constructive way of identifying ``good'' and ``bad'' basis elements in this case, in order to formulate an analysis as in Section \ref{subsec:ConvAnalysisNotFullyObsAndCont}.
%	If the system is not fully controllable, one idea could be to assemble a vectorized matrix representation of the Lie algebra $L$.
%	This is done by computing all iterated, vectorized commutators of $A_\circ$ and $B$, and combine them together as columns of a matrix until the matrix is full or no new linearly independent columns can be found.
%	However, a more detailed analysis of this approach is beyond the scope of this paper.
%\end{remark}

\begin{remark}
		The analysis presented in the following sections can be applied to the case where the matrix $A=A_\star$ is assumed to be known and $B=B(\al):=\sum_{j=1}^K\al_jB_j$ is unknown and to be identified.
		The main differences in the case of the identification of $B$ is that the state equation is linearized around an initial guess $B_\circ$, leading to
		\begin{equation*}
		\resizebox{.91\hsize}{!}{$\begin{cases} 
			\dot{\delta\ypmb}_{\circ}(t)=(A+\epsilon(t)B_\circ)\delta\ypmb_\circ(t)+\sum_{j=1}^{K}\delta\al_j \epsilon(t)B_j\ypmb_\circ(t),\quad t \in(0,T],\quad \delta\ypmb_\circ(0)=0, \\
			\dot{\ypmb}_\circ(t)=(A+\epsilon(t)B_\circ)\ypmb_\circ(t),\quad t \in(0,T],\quad \ypmb_\circ(0)=\ypmb^0.
			\end{cases}$}
		\end{equation*}
		Assumption \ref{assum:Bilinear} would be the same, only with $A$ instead of $A_\circ$ and $B_\circ$ instead of $B$.
		Notice that, in this case, we also cover Schr\"odinger-type systems of the form
		\begin{equation*}
			i\dot{\psipmb}(t)=(H+\epsilon(t)\mu_\star)\psipmb(t), \; t \in(0,T], \quad\psipmb(0)=\psipmb^0,
		\end{equation*}
		as considered in \cite{madaysalomon}, for Hermitian matrices $H,\mu_\star\in\mathbb{C}^{N\times N}$.
		This can be seen by writing $\psipmb=\psipmb_R+i\psipmb_I$, $\psipmb^0=\psipmb^0_R+i\psipmb^0_I$, $H=H_R+iH_I$ and $\mu_\star=\mu_{\star,R}+i\mu_{\star,I}$, to get
		\begin{equation}\label{eq:ODESchroedingerReal}
			\dot{\ypmb}(t)=\Bigg(\underbrace{\begin{bmatrix}
			H_I&H_R\\-H_R&H_I
			\end{bmatrix}}_{=:A}+\epsilon(t)
			\underbrace{\begin{bmatrix}
			\mu_{\star,I}&\mu_{\star,R}\\-\mu_{\star,R}&\mu_{\star,I}
			\end{bmatrix}}_{:=B_\star}\Bigg)\ypmb(t),
		\end{equation}
		for $\ypmb(t):=\begin{bmatrix} \psipmb_R(t)&\psipmb_I(t)\end{bmatrix}^\top$ and skew-symmetric matrices $A,B_\star\in\mbbr^{N\times N}$ (compare also \cite[Section 2.12.2]{LibroQuantum}).
\end{remark}

\subsection{Analysis for skew-symmetric bilinear systems}
We show in this section that Assumption \ref{assum:Bilinear} is a sufficient condition for the GN matrix $\Wcirc$, defined as in \eqref{eq:Wwidehatcirc_ij}, to be PD if the controls generated by Algorithm~\ref{algo:General} are used.
The idea of the analysis is similar to the one considered in section \ref{sec:LinearReconstruction}, meaning that we first have to show the existence of a control that makes the cost function of \eqref{eq:SplittingStepGeneral} strictly positive.
\begin{lemma}[GR initialization and splitting steps (bilinear systems)]\label{lem:nonzerosplitbilin}
	Let the matrices $A_\circ$, $B$ and $C$ satisfy Assumption \ref{assum:Bilinear}. Let $\widetilde{A}\in\spantxt (\mathcal{A})$ be an arbitrary matrix. If $T>0$ is sufficiently large, then any solution $\widetilde{\epsilon}$ to the problem
	\begin{align*}
	\max_{\epsilon\in E_{ad}} \, \norm{C\delta\ypmb_\circ\klam{\widetilde{A},\epsilon;T}}_2^2, \;
	s.t.\;\;\dot{\delta\ypmb}_{\circ}(t)&=(A_\circ+\epsilon(t)B)\delta\ypmb_\circ(t)+\widetilde{A}\ypmb_\circ(t), \quad \delta\ypmb_\circ(0)=0,\\
	\dot{\ypmb}_\circ(t)&=(A_\circ+\epsilon(t)B)\ypmb_\circ(t), \quad \ypmb_\circ(0)=\ypmb^0,
	\end{align*}
	satisfies $\norm{C\delta\ypmb_\circ\klam{\widetilde{A},\widetilde{\epsilon};T}}_2^2>0$.
	%\begin{equation}\label{eq:NonzeroSplittingBilinear}
	%\norm{C\delta\ypmb_\circ\klam{\widetilde{A},\widetilde{\epsilon};T}}_2^2>0.
	%\end{equation}
\end{lemma}
\begin{proof}
	It is sufficient to construct an $\widehat{\epsilon}_{c}\in E_{ad}$ such that $C\delta\ypmb_\circ(\widetilde{A},\widehat{\epsilon}_{c};T)\neq~0$ for $T$ sufficiently large.
	Let us define %$\widehat{\epsilon}_c$ as
	%\begin{equation*}
	$\widehat{\epsilon}_{c}(s):=\begin{cases}
	\widehat{\epsilon}(s),& \textnormal{for }0\leq s\leq \widehat{t},\\
	c,&\textnormal{for }\widehat{t}< s\leq T,
	\end{cases}$
	%\end{equation*}
	where $c\in\mbbr$, $\widehat{\epsilon}\in E_{ad}$, $T>0$ and $\widehat{t}\in(0,T)$ are to be chosen.
	Since $\widetilde{A}\neq 0$, there exists $\vpmb\in \mbbr^N, \|\vpmb\|_2=\|\ypmb^0\|_2$ such that $\widetilde{A}\vpmb\neq0$.
	By the first and second part of Assumption \ref{assum:Bilinear}, we know that \eqref{eq:BilinearSystem} is controllable on the sphere of radius $\|\ypmb^0\|_2$, meaning that there exist $\widehat{t}>0$ and $\widehat{\epsilon}\in E_{ad}$ such that $\ypmb_\circ(\widetilde{A},\widehat{\epsilon};\widehat{t})=\vpmb$.
	Defining $A_{c}:=A_\circ+cB$, we notice that $f_{\vpmb}(t):=\widetilde{A}e^{tA_c}\vpmb$ is analytic in $t$, and since $f_{\vpmb}(0)=\widetilde{A}\vpmb\neq0$, it is not equal to zero everywhere and therefore has only isolated roots, see, e.g., \cite[Theorem 10.18]{rudin}.
	Recalling that exponential matrices are always invertible (see, e.g., \cite[Theorem 2.6.38]{horn_johnson_1991}), we obtain that there exists $t_1>0$ such that $e^{-t_1(A_c)}\widetilde{A}e^{(t_1-\widehat{t})A_c}\vpmb\neq0$.
	By defining $\wpmb:=\delta\ypmb_\circ(\widetilde{A},\widehat{\epsilon};\widehat{t})$ and $\gpmb(t):=\int_{\widehat{t}}^{t}e^{-s(A_c)}\widetilde{A}e^{(s-\widehat{t})A_c}\vpmb ds+e^{-\widehat{t}A_c}\wpmb$, we observe that $\frac{d\gpmb(t_1)}{dt}=e^{-t_1(A_c)}\widetilde{A}e^{(t_1-\widehat{t})A_c}\vpmb\neq0$.
	Since $\frac{d\gpmb(t)}{dt}$ is analytic in $t$, the same holds for $\gpmb(t)$,\footnote{This follows directly from the fundamental theorem of calculus.} and since $\frac{d\gpmb(t_1)}{dt}\neq 0$ we obtain that $\gpmb(t)$ has only isolated roots.
	Notice that
	\begin{equation*}
	e^{-tA_c}\delta\ypmb(\widetilde{A},\widehat{\epsilon}_c;t) = e^{-tA_c}\int_{\widehat{t}}^{t}e^{(t-s)(A_c)}\widetilde{A}e^{(s-\widehat{t})A_c}\vpmb ds+e^{(t-\widehat{t})A_c}\wpmb 
	=\gpmb(t),
	\end{equation*}
	for $t>\widehat{t}$. Thus, it remains to show that there exists $T>\widehat{t}$ such that $Ce^{TA_c}\gpmb(T)\neq 0$.
	Assumption \ref{assum:Bilinear} guarantees that there exists $c\in\mbbr$ such that the observability matrix $\mathcal{O}_N(C,A_\circ+cB)$ has full rank.
	Hence, for any $\upmb\in\mbbr^N\setminus\{0\}$ there exists a $t_{\upmb}>\widehat{t}$ such that $Ce^{t_{\upmb}A_c}\upmb\neq0$.
	Since $t \mapsto Ce^{tA_c}\upmb$ is analytic in $t$, $Ce^{t_{\upmb}A_c}\upmb\neq0$ implies that it has only isolated roots.
	Thus, for $t>\widehat{t}$, $t\mapsto Ce^{tA_c}\gpmb(t)$ is the composition of two analytic functions which both have only isolated roots, and is therefore also analytic with isolated roots.
	Hence, there exists $T>\widehat{t}$ such that $C\delta\ypmb(\widetilde{A},\widehat{\epsilon}_c;T)=Ce^{TA_c}\gpmb(T)\neq 0$.
\end{proof}

Now, we can prove our main result, whose proof is the same as the one of Theorem \ref{thm:convergence_fully_obscont}, in which Lemma \ref{lem:nonzerosplitbilin}  has to be used instead of Lemma \ref{lem:positivitymaxproblems}.

\begin{theorem}[positive definiteness of the GN matrix $\widehat{W}_\circ$ (bilinear systems)]\label{thm:ConvergenceBilinear}
	Let $\alcirc\in\mbbr^K$ be such that the matrices $A(\alcirc),B\in\mathfrak{so}(N)$ and $C\in\mathbb{R}^{P \times N}$ satisfy Assumption \ref{assum:Bilinear}.
	For $K\leq N^2$, let $\mathcal{A}=\{A_1,\ldots,A_K\}\subset~\mathfrak{so}(N)$ be a set of linearly independent matrices such that $\Astar\in\spantxt\;\mathcal{A}$, and let $\{\epsilon^1,\ldots,\epsilon^{K} \}\subset E_{ad}$ be controls generated by Algorithm~\ref{algo:General}. 
	Then the GN matrix $\widehat{W}_\circ$, defined in \eqref{eq:Wwidehatcirc_ij}, is PD.
\end{theorem}
%The proof of Theorem \ref{thm:ConvergenceBilinear} is the same as the one of Theorem \ref{thm:convergence_fully_obscont}, only using Lemma \ref{lem:nonzerosplitbilin} instead of Lemma \ref{lem:positivitymaxproblems}.

\subsection{Positive definiteness of the GN matrix}\label{subsec:PosDefGaussNewtonBilinear}
As in section \ref{subsec:PosDefGaussNewtonLinear}, we show that if the GN matrix in $\alcirc$ is PD, then the same is true locally, for all iterates $\al_c$ of GN.
We start by writing the matrix $\widehat{W}(\al)$ as a function of $\al$:
\begin{equation}\label{eq:WijalBilinear}
[\widehat{W}(\al)]_{i,j}:=\sum_{m=1}^{K}\langle C\delta\ypmb(\al,A_i,\epsilon^m;T),C\delta\ypmb(\al,A_j,\epsilon^m;T)\rangle,\quad i,j\in\{1,\ldots,K\},
\end{equation}
where $\delta\ypmb(\al,\widehat{A},\epsilon;T)$ denotes the solution at time $T$ of
\begin{equation}\label{eq:LinearizedODEGeneral}
\begin{cases}
\dot{\delta\ypmb}(t)&=(A(\al)+\epsilon(t)B)\delta\ypmb(t)+\widehat{A}\ypmb(t), \quad \delta\ypmb(0)=0,\\
\dot{\ypmb}(t)&=(A(\al)+\epsilon(t)B)\ypmb(t), \quad \ypmb(0)=\ypmb^0.
\end{cases}
\end{equation}
Now, we want to prove the same positive definiteness result as in Lemma \ref{lem:ApproximationWhatalpha}.
\begin{lemma}[positive definiteness of $\widehat{W}_\circ$ (bilinear systems)]\label{lem:ApproxWhatalphaBilinear}
	Let $\widehat{W}_\circ$, defined in \eqref{eq:Wwidehatcirc_ij}, be PD and denote by $\sigma^\circ_K>0$ the smallest singular value of $\widehat{W}_\circ$. Then, there exists $\delta:=\delta(\sigma^\circ_K)>0$ such that for any $\al\in\mbbr^K$ with $\|\al-\alcirc\|_2<\delta$, the matrix $\widehat{W}(\al)$, defined as in \eqref{eq:WijalBilinear}, is also PD.
\end{lemma}
\begin{proof}
	Recalling the proof of Lemma \ref{lem:ApproximationWhatalpha}, it is sufficient to show that the solution $\delta\ypmb(\al,\widehat{A},\epsilon;T)$ of \eqref{eq:LinearizedODEGeneral} is continuous in $\al$.
	By \cite[Proposition 3.26]{LibroQuantum},\footnote{This result is a special case of the implicit function theorem; see, e.g., \cite[Theorem 3.4]{LibroQuantum}.} we obtain continuity of the map $\al\mapsto\ypmb(A(\al),\epsilon;T)$ and analogously the continuity of $\al\mapsto\delta\ypmb(\al,\widehat{A},\epsilon;T)$.
\end{proof}
Using the result from Lemma \ref{lem:ApproxWhatalphaBilinear}, we can directly prove our main result.
\begin{theorem}[convergence of GN (bilinear systems)]
	Let $\alcirc\in\mbbr^K$ be such that the matrices $A(\alcirc)$, $B$ and $C$ satisfy Assumption \ref{assum:Bilinear}, and let $(\epsilon^m)_{m=1}^K\subset E_{ad}$ be generated by Algorithm~\ref{algo:General}.
	Denote by $\widehat{\sigma}_K$ the smallest singular value of $\Wcirc$, defined in \eqref{eq:Wwidehatcirc_ij}.
	Then there exists $\delta=\delta(\widehat{\sigma}_K)>0$ such that, if $\alstar\in\mbbr^K$ satisfies $\|\alstar-\alcirc\|\leq\delta$, then GN for the solution \eqref{eq:IdentificationProblemBilinear}, initialized with $\alcirc$, converges to $\alstar$.
\end{theorem}
\begin{proof}
	Theorem \ref{thm:ConvergenceBilinear} guarantees that $\Wcirc$ is PD, meaning that $\widehat{\sigma}_K>0$.
	Analogously to the proof of Lemma \ref{lem:ApproxWhatalphaBilinear}, one can also show that the functions $R_m(\al)$, defined in \eqref{eq:rm}, are Lipschitz continuously differentiable in $\al$ for all $m\in\{1,\ldots,K\}$.
	Thus, the result follows by Lemma \ref{lem:ApproxWhatalphaBilinear}.
\end{proof}

\subsection{Local uniqueness of solutions}
Let us study the local properties of problem \eqref{eq:IdentificationProblemBilinear} around $\alstar$.
We use the same approach as in the linear case, and start by rewriting problem \eqref{eq:IdentificationProblemBilinear} in a matrix-vector form.
\begin{lemma}[online identification problem in matrix form (bilinear systems)]\label{lem:IdentInMatrixBilinear}
	Problem \eqref{eq:identificationproblem} is equivalent to
	\begin{equation*}
	\min_{\al\in\mbbr^K}
	\frac{1}{2}\langle\alstar-\al,\widetilde{W}\klam{\alstar,\al}(\alstar-\al)\rangle,
	\end{equation*}
	where $\widetilde{W}\klam{\alstar,\al}\in\mbbr^{K\times K}$ is defined as $\widetilde{W}\klam{\alstar,\al} = \sum_{m=1}^{K}W\klam{\alstar,\al,\epsilon^m}$ with
	\begin{equation*}
	[W\klam{\alstar,\al,\epsilon^m}]_{i,j}:=\langle C\delta\ypmb_m(\alstar,\al,A_j;T),C\delta\ypmb_m(\alstar,\al,A_j;T)\rangle,
	\end{equation*}
	for $i,j\in\{1,\ldots,K\}$ and where $C\delta\ypmb_m(\alstar,\al,A;T)$ is the solution at time $T$ of
	\begin{equation*}
	\begin{cases}
	\dot{\delta\ypmb}(t)&=(A(\alstar)+\epsilon^m(t)B)\delta\ypmb(t)+A\ypmb(t), \quad \delta\ypmb(0)=0,\\
	\dot{\ypmb}(t)&=(A(\al)+\epsilon^m(t)B)\ypmb(t), \quad \ypmb(0)=\ypmb^0.
	\end{cases}
	\end{equation*}
\end{lemma}
The proof of Lemma \ref{lem:IdentInMatrixBilinear} is analogous to the one of Lemma \ref{lem:identificationinmatrixform} and we omit it here for brevity.
%%%%% SUPPL
% (for details see the supplementary material \cite{supp_material}).
%%%%% SUPPL
Notice that the notations in \eqref{eq:WijalBilinear} and Lemma \ref{lem:IdentInMatrixBilinear} are related in the sense that $\widehat{W}(\al)=\widetilde{W}(\al,\al)$. 
Now, proceeding as in Section \ref{subsec:LocUniqSolLinear} and defining the set of all global solutions $\mathcal{S}_{global}:= \Big\{\al\in\mbbr^K : (\alstar-\al)\in\kertxt\; \widetilde{W}(\alstar,\al) \Big\}$, we obtain the same local uniqueness of the solution $\alstar$ to \eqref{eq:IdentificationProblemBilinear}, meaning that if $\widehat{W}(\alstar)=\widetilde{W}(\alstar,\alstar)$ is PD, the same holds for $\widetilde{W}(\alstar,\al)$ when $\al$ is close to $\alstar$.

%\newpage
\section{Towards general nonlinear GR algorithms}\label{sec:nonlinear}
The LGR algorithm introduced in the previous sections only considers the linearized system.
Thus it does not have access to the full (nonlinear) dynamics and can only capture the local characteristics of the considered system. 
Moreover, as we will show in section~\ref{sec:numerics}, the standard GR algorithm can outperform LGR when $\al_\circ$ is far from the solution. 
However, the analysis of LGR allows us to better understand the local behavior of GR and prove that locally it is capable to construct control functions that guarantee convergence of GN. 
This analysis is carried out in section~\ref{sec:nonlinear:analysis}. 
This is the first analysis of GR algorithms for nonlinear problems. 
While section~\ref{sec:nonlinear:analysis} focuses on GR, we also briefly discuss its optimized version called optimized GR (OGR), introduced in \cite{siampaper}, and propose a slight improvement of the original version.

\subsection{A local analysis for nonlinear GR algorithms}\label{sec:nonlinear:analysis}
This section is concerned with general nonlinear systems of the form $\dot{\ypmb}(t)=f(A(\al^0)+A(\delta\al_\star),\ypmb(t),\eps(t))$ with the goal of reconstructing $A(\delta\al_\star)=\Astar-A(\al^0)$.
Here, the shift of $\Astar$ is considered to perform a local analysis near $A(\al^0)$.
The goal is to prove convergence of GN for the controls generated by the GR Algorithm \ref{algo:GeneralNonlin} using a local analogy to Algorithm \ref{algo:General}.
\begin{algorithm}[t]
	\caption{Nonlinear Greedy Reconstruction Algorithm}
	\begin{algorithmic}[1]\label{algo:GeneralNonlin} 
	\begin{small}
		\REQUIRE A set of linearly independent operators $\mathcal{A}=\{A_1,\ldots,A_{K}\}$, an (initial) operator $A(\alcirc)\in\spantxt\;\mathcal{A}$ and a family of compact sets $\mathcal{K}_j \subset \mathbb{R}^j$, $j=1,\dots,K-1$.
		\STATE Compute the control $\eps^1$ by solving \begin{equation}\label{eq:InitGeneralNonlin}
		\max_{\eps\in E_{ad}} \norm{C\ypmb(A(\alcirc),\eps;T)-C\ypmb(A(\alcirc) + A_1,\eps;T)}_2^2.
		\end{equation}
		\FOR{ $k=1,\dots, K-1$ }
		\STATE \underline{Fitting step}: $A^{(k)}(\bet):=\sum_{j=1}^{k}\bet_jA_j$, find $\bet=(\bet^{k}_j)_{j=1,\dots,k}$ that solves 
		\begin{equation}\label{eq:FittingGeneralNonlin}
		\min_{\bet\in\mathcal{K}_k}\sum_{m=1}^{k}\norm{C\ypmb(A(\alcirc)+A^{(k)}(\bet),\eps^m;T)-C\ypmb(A(\alcirc)+A_{k+1},\eps^m;T)}_2^2.
		\end{equation}
		\STATE \underline{Splitting step}: Find $\eps^{k+1}$ that solves \begin{equation}\label{eq:SplittingGeneralNonlin}
		\max_{\eps\in E_{ad}}\norm{C\ypmb(A(\alcirc)+A^{(k)}(\bet^{k}),\eps;T)-C\ypmb(A(\alcirc)+ A_{k+1},\eps;T)}_2^2.
		\end{equation}
		%\STATE Update $k \leftarrow k+1$.
		\ENDFOR
	\end{small}
	\end{algorithmic}
\end{algorithm}
\setlength{\textfloatsep}{8pt}
Notice that there are a few differences between Algorithms \ref{algo:GeneralNonlin} and \ref{algo:General}.
To derive a local analogy between them, all operators from the set $\mathcal{A}$ are shifted by $A(\alcirc)$.
Additionally, the fitting step problem \eqref{eq:FittingGeneralNonlin} only minimizes over a compact set $\mathcal{K}_k\subset \mathbb{R}^k$.
However, this is not restrictive since the set $\mathcal{K}_k$ can be chosen arbitrarily large.
Finally, the initialization problem \eqref{eq:InitGeneralNonlin} is different from the initialization \eqref{eq:InitializationGeneral}.
This is due to results obtained in \cite{siampaper} which suggest that one should not simply maximize the state corresponding to the first element $A_1$ in the set, but rather maximize the difference to the state that is observed when no elements from $\mathcal{A}$ are considered.
%which implies that any solution $\beta^k$ to problem \eqref{eq:FittingGeneralNonlin} is also small.

We recall that, in order to obtain our main results for Algorithm~\ref{algo:GeneralNonlin}, it is sufficient to prove two points.
First, that the fitting step \eqref{eq:FittingGeneralNonlin} identifies the kernel of the submatrix $[\widehat{W}_\circ^{(k)}]_{[1:k+1,1:k+1]}$.
Second, that for the initialization and each splitting step in Algorithm \ref{algo:GeneralNonlin} there exists at least one control for which the corresponding cost function is strictly positive (making the submatrix $[\widehat{W}_\circ^{(k+1)}]_{[1:k+1,1:k+1]}$ PD).

To prove the fitting step result, we need some continuity properties of the argmin operator.
For this purpose, we introduce the following definition of hemi-continuous set-valued correspondences (see, e.g., \cite[Chapter VI,$\S$1]{nla.cat-vn740512}).
\begin{definition}[hemi-continuity]\label{def:UpperHemiContinuous}
	Let $X\subset\mbbr$ be an open interval.
	A set-valued correspondence $c:X\rightrightarrows \mbbr^k$ is called upper hemi-continuous (u.h.c.) if for each $x_0\in X$ and each open set $G\subset \mbbr^k$ with $c(x_0)\subset G$ there exists a neighborhood $U(x_0)$ such that $x\in U(x_0)\Rightarrow c(x)\subset G,$ and called lower hemi-continuous (l.h.c.) if for each $x_0\in X$ and each open set $G\subset \mbbr^k$ meeting $c(x_0)$ there exists a neighborhood $U(x_0)$ such that $x\in U(x_0)\Rightarrow c(x)\cap G\neq\emptyset.$
	Furthermore, $c:X\rightrightarrows \mbbr^k$  is called   hemi-continuous if it is u.h.c. and l.h.c.
\end{definition}
Using Definition \ref{def:UpperHemiContinuous}, we can recall the Berge maximum theorem \cite[Theorem 17.31]{aliprantis06}.
\begin{lemma}[Berge maximum theorem]\label{lem:BergeMaximumTheorem}
	Let $X\subset\mbbr$ be an open interval.
	Let $J:\mbbr^k\times X\rightarrow\mbbr$ be a continuous function and $\phi:X\rightrightarrows\mbbr^k$ be a hemi-continuous, set-valued correspondence such that $\phi(x)$ is nonempty and compact for any $x\in X$.
	Then the correspondence $c:X\rightrightarrows\mbbr^k$ defined by $c(x):=\argmin\limits_{z\in\phi(x)}J(z;x)$ is u.h.c.
\end{lemma}
We will also need the following technical lemma.
\begin{lemma}[limit of set-valued correspondance]\label{lem:UHCSequenceConverges}
	Let $X\subset\mbbr$ be an open interval with $0\in X$, and $c:X\rightrightarrows\mbbr^k$ be a u.h.c. correspondence.
	If $c(0)=\{0\}$, then $\lim_{k\rightarrow\infty} c(x_k)=\{0\}$ for any sequence $\{x_k\}_{k=1}^\infty$ such that $\lim_{k\rightarrow\infty} x_k=0$.
	%If $0\in c(0)$ is an isolated point, then $\lim_{k\rightarrow\infty} c(x_k)=\{0\}$ for any sequence $\{x_k\}_{k=1}^\infty$ with $\lim_{k\rightarrow\infty} x_k=0$.
\end{lemma}
\begin{proof}
	Consider an arbitrary sequence $\{x_k\}_{k=1}^\infty$ with $\lim_{k\rightarrow\infty} x_k=0$, and let $c(0)=\{0\}$. %$0\in c(0)$ be an isolated point.
	%It is sufficient to show that there exists a sequence $\{\epsilon_k\}_{k=1}^\infty\subset\mbbr$ that satisfies $\lim_{k\rightarrow\infty} \epsilon_k=0$ and $c(x_k)\subset (-\epsilon_k,\epsilon_k)$.
	It is sufficient to show that for any $\epsilon>0$ there exists $n_\epsilon\in\mbbn$ such that for all $k\geq n_\epsilon$ we have $c(x_k)\subset \mathcal{B}^k_\epsilon(0)$.
	Let $\epsilon>0$ and define $G_\epsilon:=\mathcal{B}^k_\epsilon(0)$.
	Since $c(0)=\{0\}$ %$0$ is an isolated point of $c(0)$ 
	and $c$ is u.h.c., there exists a neighborhood $U_\epsilon(0)\subset \mbbr$ such that $c(x)\subset G_\epsilon$ for any $x\in U_\epsilon(0)$.
	Since $U_\epsilon(0)$ is an open neighborhood of $0$, there exists $\xi_\epsilon>0$ such that $(-\xi_\epsilon,\xi_\epsilon)\subset U_\epsilon(0)$.
	Since $\lim_{k\rightarrow\infty} x_k=0$, there exists $n_\epsilon$ such that for all $k\geq n_\epsilon$ we have $x_k\in (-\xi_\epsilon,\xi_\epsilon)$ and hence $c(x_k)\subset\mathcal{B}^k_\epsilon(0)$.
\end{proof}
To use Lemmas \ref{lem:BergeMaximumTheorem} and \ref{lem:UHCSequenceConverges}, we make the following assumptions.
\begin{assumption}\label{assum:LocalFitting}
	Let $k\in\{1,\ldots,K-1\}$ and define,
	\begin{equation*}
	    J_k(\bet;A_{k+1}):=\sum_{m=1}^{k}\|C\ypmb(A(\alcirc)+A^{(k)}(\bet),\eps^m;T)-C\ypmb(A(\alcirc)+A_{k+1},\eps^m;T)\|_2^2.
	\end{equation*}
	\begin{itemize}
		\item If $\|A_{k+1}\|$ is small enough, then there exists a $\bet^k=\bet^k(A_{k+1})$ that solves \eqref{eq:FittingGeneralNonlin} with $J_k(\bet^k;A_{k+1})=0$.
		\item There exists $\nu>0$ such that $\mathcal{B}^k_\nu(0)\subset \mathcal{K}_k$ and $\argmin_{\bet\in\overline{\mathcal{B}^k_\nu(0)}}J_k(\bet;0)=\{0\}$.
	\end{itemize}
\end{assumption}

%\begin{remark}
	%Proof idea for the second point: Consider $f(\bet):=C\ypmb(A^{(k)}(\bet),\eps^m;T)$.
	%Then $f'(\al_\circ)(\delta\bet)=C\delta\ypmb_\circ(A^{(k)}(\delta\bet),\eps^m;T)$. \comment{I do not remember what I intended to do here. I will double check it and come back to it.}
%\end{remark}

The first point in Assumption \ref{assum:LocalFitting} guarantees that locally near $A(\al_{\circ})$, for $\|A_{k+1}\|$ small enough, one can solve \eqref{eq:FittingGeneralNonlin} making the cost function zero, meaning that one can find a linear combination of the first $k$ elements for which the final state cannot be distinguished from the $k+1$-th element by any of the $k$ computed controls.
On the other hand, if the minimum function value is strictly positive, then there already exists a control in the set $(\eps_m)_{m=1}^k$ that discriminates (splits) these two states.

The second point in Assumption \ref{assum:LocalFitting} ensures that $\{0\}=\argmin _{\bet\in\overline{\mathcal{B}^k_\nu(0)}}J_k(\bet,0)$.
%$0$ is an isolated point of $\argmin _{\bet\in\mathcal{B}^k_\nu(0)}J(\bet,0)$.
If this was not true, it would mean that, for any radius $\nu>0$, the ball $\mathcal{B}^k_\nu(0)$ would contain infinitely many $\bet\in\mbbr^k\setminus\{0\}$ satisfying $J_k(\bet,0)=0$. %$\sum_{m=1}^{k}\|C\ypmb(A(\alcirc)+A^{(k)}(\bet),\eps^m;T)-C\ypmb(A(\alcirc),\eps^m;T)\|_2^2=0$.
Hence, for an infinite number of linear combinations in the set $\{A_1,\ldots,A_k \}$, the corresponding states could not be distinguished by any of the previously selected controls.
However, this implies that at least one of the previous splitting steps was not successful, which contradicts what we assume to reach iteration $k$.

Now, we can show that the local nonlinear fitting step problem \eqref{eq:FittingGeneralNonlin} is able to identify the kernel of the submatrix $[\widehat{W}_\circ^{(k)}]_{[1:k+1,1:k+1]}$, if it exists.
\begin{theorem}[nonlinear GR fitting step problems]\label{thm:FittingNonlinearToLinearized}
	Let $k\in\{1,\ldots,K\}$ and let $\bet^k$ be a solution to \eqref{eq:FittingGeneralNonlin}.
	If $\|A_{k+1}\|$ is sufficiently small and Assumption \ref{assum:LocalFitting} holds, then $\bet^k$ also solves \eqref{eq:FittingStepGeneral} with $$\sum_{m=1}^{k}\|C\delta\ypmb_\circ(A^{(k)}(\bet^k),\eps^m;T)-C\delta\ypmb_\circ(A_{k+1},\eps^m;T)\|_2^2=0.$$
\end{theorem}
\begin{proof}
    \comment{Let $A_{k+1} = \delta_k\widetilde{A}_{k+1}$, where $\delta_k = \| A_{k+1} \|$ and $\| \widetilde{A}_{k+1} \|=1$.}
    Now, define $\widehat{J}_k(\bet,\delta_k):=J_k(\bet,\delta_k \comment{\widetilde{A}_{k+1}})$.
    The first point of Assumption \ref{assum:LocalFitting} implies that there exists a $\widehat{\delta}_k>0$ such that for all $|\delta_k|<\widehat{\delta}_k$ we have $\widehat{J}_k(\bet,\delta_k)=0$.
	Thus, Lemma \ref{lem:BergeMaximumTheorem} guarantees that the correspondence $c_k: (-\widehat{\delta}_k,\widehat{\delta}_k) \rightrightarrows\mbbr^k,c_k(\delta_k)=\argmin_{\bet\in\mathcal{K}_k}\widehat{J}_k(\bet;\delta_k)$ is u.h.c.\footnote{Note that, in this setting, the correspondence $\phi:(-\widehat{\delta}_k,\widehat{\delta}_k)\rightrightarrows\mbbr^k$ mentioned in Lemma \ref{lem:BergeMaximumTheorem} is defined as $\phi(x)=\mathcal{K}_k$ for any $x\in(-\widehat{\delta}_k,\widehat{\delta}_k)$ with $\mathcal{K}_k$ compact, and is therefore hemi-continuous.}
	
	According to the second point of Assumption~\ref{assum:LocalFitting}, $c_k(0)=0$ is an isolated solution of \eqref{eq:FittingGeneralNonlin}.
	Hence, the upper hemi-continuity of $c_k$ guarantees that for $\delta_k\rightarrow0$ we have $\bet^k\rightarrow0$ for any corresponding solution $\bet^k=\bet^k(\delta_k)$ of \eqref{eq:FittingGeneralNonlin}.
	
	Now, let $m\in\{1,\ldots,k\}$.
	If $\widehat{J}_k(\bet^k;\delta_k)=0$, then 
	\begin{equation}\label{eq:Eq1}
	C\ypmb(A(\alcirc)+A^{(k)}(\bet^k),\eps^m;T)-C\ypmb(A(\alcirc)+\delta_k \comment{\widetilde{A}_{k+1}},\eps^m;T)=0.
	\end{equation}
	We define $g(\al):=C\ypmb(A(\al),\eps^m;T)$.
	Since $f(A,\ypmb,\eps)$ in \eqref{eq:ODE} is assumed to be differentiable with respect to $A$ and $\ypmb$, we obtain that the map $A\mapsto\ypmb(A,\eps;T)$ is differentiable with respect to $A$ by the implicit function theorem (see, e.g., \cite[Theorem 17.13-1]{Ciarlet}).
	Hence, $C\ypmb(A(\al),\eps;T)$ is also differentiable with respect to $\al$.
	By Taylor's theorem, we get $g(\alcirc+\vpmb)=g(\alcirc)+g'(\alcirc)(\vpmb)+O(\|\vpmb\|_2^2)$ for $\vpmb\in\mbbr^k$.
	Defining $\widehat{\bet^k}$ and $\widehat{\del}_k$ as $\widehat{\bet^k}:=[\bet^k,0,\cdots,0]^\top\in\mbbr^k$ and $\widehat{\del}_k:=[0,\cdots,0,\delta_k]^\top\in\mbbr^k$, we can rewrite \eqref{eq:Eq1} as 
	\begin{align*}
	0=g(\alcirc+\widehat{\bet^k})-g(\alcirc+\widehat{\del}_{k+1})&=g'(\alcirc)(\widehat{\bet^k})-g'(\alcirc)(\widehat{\del}_{k+1})+O(\|\widehat{\bet^k}\|_2^2)+O(|\delta_k|^2).
	\end{align*}
	Since $g'(\alcirc)(\widehat{\bet^k})=C\delta\ypmb_\circ(A^{(k)}(\bet^k),\eps^m;T)$ and $g'(\alcirc)(\widehat{\del}_{k+1})=C\delta\ypmb_\circ(\delta_k \comment{\widetilde{A}_{k+1}},\eps^m;T)$, we obtain
	\begin{equation}\label{eq:orders}
	0= C\delta\ypmb_\circ(A^{(k)}(\bet^k),\eps^m;T)-C\delta\ypmb_\circ(\delta_k \comment{\widetilde{A}_{k+1}},\eps^m;T)+O(\|\widehat{\bet^k}\|_2^2)+O(|\delta_k|^2).
	\end{equation}
	Since $\bet^k=\bet^k(\delta_k)\rightarrow0$ for $\delta_k\rightarrow0$, we know that all four terms vanish for $\delta_k\rightarrow0$.
	However, $O(|\delta_k|^2)$ converges faster than $C\delta\ypmb_\circ(\delta_k \comment{\widetilde{A}_{k+1}},\eps^m;T)$ and $O(\|\widehat{\bet^k}\|_2^2)$ faster than $C\delta\ypmb_\circ(A^{(k)}(\bet^k),\eps^m;T)$.
	Hence, \eqref{eq:orders} can only be true for $\delta_k\rightarrow0$ if \linebreak $C\delta\ypmb_\circ(A^{(k)}(\bet^k),\eps^m;T)-C\delta\ypmb_\circ(\delta_k \comment{\widetilde{A}_{k+1}},\eps^m;T)=0$ for $\delta_k$ small enough, which is equivalent to $C\delta\ypmb_\circ(A^{(k)}(\bet^k),\eps^m;T)-C\delta\ypmb_\circ(A_{k+1},\eps^m;T)=0$ for $\|A_{k+1}\|$ sufficiently small.
\end{proof}
%\begin{remark}
%	As we can see in Theorem \ref{thm:FittingNonlinearToLinearized}, the local performance of Algorithm \ref{algo:GeneralNonlin} refers to the setting, where we consider a system of the form $\dot{\ypmb}(t)=f(A(\al^0)+A(\delta\alpha_\star),\ypmb(t),\eps(t))$ and wish to reconstruct $A(\delta\alpha_\star)=\Astar-A(\al^0)$.
%\end{remark}
\comment{
Let us now comment about the assumption on $\|A_{k+1}\|$.
The goal of this section is to prove a local result around the approximation $A(\al_\circ)$. 
Since the map $\al \mapsto A(\al)$ is linear, the terms $A(\al_\circ) + A^{(k)}(\bet^k)$ and $A(\al_\circ) + A_{k+1}$ are perturbations of $A(\al_\circ)$. 
These are clearly used in the proof of Theorem \ref{thm:FittingNonlinearToLinearized}.
In order to remain close to $A(\al_\circ)$, we require that the basis elements $A_{k+1}$ are sufficiently small. 
However, this is not a restrictive assumption since it can be interpreted as a simple rescaling of the operators $A_1,\dots,A_K$, which are only required to be linearly independent.
Thus, the sufficiently small assumption of $\| A_{k+1} \|$ has to be understood in the sense that one has to perturb $A(\al_\circ)$ by stepping in direction of the (linearly independent) element $A_{k+1}$. 
However, the length of the step must be sufficiently small to remain in a neighborhood of $A(\al_\circ)$ (related to the Taylor expansion used in the proof of Theorem \ref{thm:FittingNonlinearToLinearized}).
}

Regarding the initialization and splitting step result, we make now the assumption that there always exists a control that makes the corresponding cost function value strictly positive, and discuss specific cases where this assumption holds.

\begin{assumption}\label{assum:LocalSplitting}
	Let $k\in\{1,\ldots,K-1\}$ and $\bet^k\in\mbbr^k$ be the solution of \eqref{eq:FittingGeneralNonlin}. 
	There exists a solution $\eps^{k+1}\in E_{ad}$ to \eqref{eq:SplittingGeneralNonlin} that simultaneously satisfies
	\begin{equation}\label{eq:AssumSplittingIneqOne}
	    \|C\ypmb(A(\alcirc)+A^{(k)}(\bet^{k}),\eps^{k+1};T)-C\ypmb(A(\alcirc)+ A_{k+1},\eps^{k+1};T)\|_2^2>0,
	\end{equation}
	and
	\begin{equation}\label{eq:AssumSplittingIneqTwo}
	    \|C\delta\ypmb_\circ(A^{(k)}(\bet^k),\eps^{k+1};T)-C\delta\ypmb_\circ(A_{k+1},\eps^{k+1};T)\|_2^2>0.
	\end{equation}
	Let \eqref{eq:AssumSplittingIneqOne}-\eqref{eq:AssumSplittingIneqTwo} also hold for a solution $\eps^1\in E_{ad}$ to \eqref{eq:InitGeneralNonlin} with $k=0$ and $\bet^0=0$.
\end{assumption}
In Theorem \ref{thm:MainResult}, we will investigate Assumption \ref{assum:LocalSplitting} for the two settings considered in sections \ref{sec:LinearReconstruction} and \ref{sec:BilinearReconstruction}.
Now, we state a result relating Algorithms \ref{algo:General} and \ref{algo:GeneralNonlin}.
\begin{theorem}[\comment{positive definiteness of the GN matrix $\widehat{W}_\circ$ (general systems)}]\label{thm:MainTheorem}
    Consider the general setting of system \eqref{eq:ODE} with a set of linearly independent matrices $\{A_1,\ldots,A_K\}$ such that $\|A_k\|$ be sufficiently small for all $k\in\{1,\ldots,K\}$. 
    Let $(\eps^m)_{m=1}^K\subset E_{ad}$ be generated by Algorithm \ref{algo:GeneralNonlin} such that Assumption \ref{assum:LocalFitting} holds for all $k\in\{1,\ldots,K-1\}$ and $\eps^m$ satisfies Assumption \ref{assum:LocalSplitting} for all $m\in\{1,\ldots,K\}$. Then the GN matrix $\widehat{W}_\circ$, defined in \eqref{eq:Wwidehatcirc_ij}, is PD.
\end{theorem}
The proof of Theorem \ref{thm:MainTheorem} is similar to that of Theorem \ref{thm:convergence_fully_obscont} and is omitted for brevity.

\comment{In the following subsections \ref{sub:sub:1} and \ref{sub:sub:2} we discuss Assumption \ref{assum:LocalSplitting} for linear and bilinear systems and for a general class of nonliner systems, respectively.}

\subsubsection{\comment{Linear and bilinear control systems}}\label{sub:sub:1}
%It remains to show that Assumption \ref{assum:LocalSplitting} holds in the settings considered in sections \ref{sec:LinearReconstruction} and \ref{sec:BilinearReconstruction}.
In this section, we discuss Assumption \ref{assum:LocalSplitting} in the settings considered in sections \ref{sec:LinearReconstruction} and \ref{sec:BilinearReconstruction}.
First, we require the following results (see, e.g., \cite[p. 1079]{Whittlesey}).
\begin{lemma}[on analytic functions in Banach spaces]\label{lem:AnalyticFunctionsSplitting}
	Let $X,Y$ denote real Banach spaces and $\mathcal{B}_r(x)\subset X$ the open ball with center $x\in X$ and radius $r>0$.
	For an open set $D\subset X$, let the functions $f,g:D\rightarrow Y$ be analytic.
	If there exist $x_f,x_g\in D$ such that $f(x_f)\neq0$ and $g(x_g)\neq0$, then for any $x\in D$ and any $r>0$ there exists a $\widetilde{x}\in \mathcal{B}_r(x)\subset D$ such that $f(\widetilde{x})\neq0$ and $g(\widetilde{x})\neq0$.
\end{lemma}
We also require the following result about the analycity of control-to-state maps, which follows directly from the implicit function theorem (see, e.g., \cite[p. 1081]{Whittlesey}).
\begin{lemma}[analycity of control-to-state maps]\label{lem:AnalyticControlToStateMaps}
	Consider system \eqref{eq:ODE} and define the map $c:U\times Y\rightarrow Z$ as  $c(\eps,\ypmb):=[\dot{\ypmb}-f(A,\ypmb,\eps),\ypmb(0)-\ypmb^0]$, where $U$ is the Hilbert space of control functions, $Y$ is the (Banach) space where solutions to \eqref{eq:ODE} lie and $Z$ is a Banach space.
	If $c$ is analytic in $\eps$ and $\ypmb$, \eqref{eq:ODE} has a unique solution $\ypmb=\ypmb(\eps)\in Y$ such that $c(\ypmb(\eps),\eps)=0$ for each $\eps\in E_{ad}\subset U$ and the linearized state equation $\dot{\delta\ypmb}=\delta_{\ypmb}f(A,\ypmb(\eps),\eps)(\delta\ypmb) - \varphi$ with $\delta\ypmb(0)=\varphi^0$ is uniquely solvable for any $[\varphi,\varphi^0]\in Z$, then the control-to-state map $L:E_{ad}\rightarrow Y,\eps\mapsto \ypmb(\eps)$ is analytic.
	If the solution space $Y$ is such that the evaluation map $S_T:Y\rightarrow\mbbr^N, \ypmb\mapsto\ypmb(T)$ is linear and continuous, then also the map $S:E_{ad}\rightarrow\mbbr^N,\eps\mapsto(\ypmb(\eps))(T)$ is analytic.
\end{lemma}
\begin{proof}
    First, we prove that the control-to-state map $L:E_{ad}\rightarrow Y,\eps\mapsto \ypmb(\eps)$ is analytic.
    This follows directly from the implicit function theorem \cite[p. 1081]{Whittlesey} if we can show that the map $D_{\ypmb}c(\eps,\ypmb)$ is an isomorphism of $Y$ on $Z$ for any pair $(\widetilde{\eps},\widetilde{\ypmb})\subset U\times Y$ such that $\widetilde{\ypmb}$ is the unique solution to \eqref{eq:ODE} for $\widetilde{\eps}$, i.e. $c(\widetilde{\eps},\widetilde{\ypmb})=0$.
    Since the equation for the derivative $D_{\ypmb}c(\widetilde{\eps},\widetilde{\ypmb})(\delta\ypmb)=\varphi$, which is equivalent to $\dot{\delta\ypmb}=\delta_{\ypmb}f(A,\widetilde{\ypmb},\widetilde{\eps})(\delta\ypmb)-\varphi$ with $\delta\ypmb(0)=\varphi^0$, admits a unique solution $\delta\ypmb\in Y$ for any $[\varphi,\varphi^0]\in Z$, $D_{\ypmb}c(\widetilde{\eps},\widetilde{\ypmb})$ is bijective and therefore an isomorphism of $Y$ on $Z$.
    
    It remains to show that also the map $S:E_{ad}\rightarrow\mbbr^N,\eps\mapsto(\ypmb(\eps))(T)$ is analytic.
    Consider an arbitrary $\eps_0\in E_{ad}$.
    Since the control-to-state map $L$ is analytic, there exist (by definition, see, e.g., \cite[p. 1078]{Whittlesey}) $\ell$-linear, symmetric and continuous maps $a_\ell:(E_{ad})^\ell\rightarrow\mbbr^N, (\eps_1,\ldots,\eps_\ell)\mapsto a_\ell(\eps_1,\ldots,\eps_\ell)$ such that $\ypmb(\eps)=\sum_{\ell=0}^{\infty}a_\ell(\eps-\eps_0)^\ell$.
    Now, define the maps $b_\ell:(E_{ad})^\ell\rightarrow\mbbr^N$ as $b_\ell(\eps)^\ell:=(a_\ell(\eps)^\ell)(T)$, meaning that $\sum_{\ell=0}^{\infty}b_\ell(\eps-\eps_0)^\ell=(\ypmb(\eps))(T)$.
    Since the evaluation map $S_T:Y\rightarrow\mbbr^N, \ypmb\mapsto\ypmb(T)$ is linear and continuous, the maps $b_\ell$ are $\ell$-linear, symmetric and continuous.
    Thus, the map $S:E_{ad}\rightarrow\mbbr^N,\eps\mapsto(\ypmb(\eps))(T)=\sum_{\ell=0}^{\infty}b_\ell(\eps-\eps_0)^\ell$ is analytic by definition.
\end{proof}
In our case, we consider $U=L^2(0,T;\mbbr^M)$ in the linear and $U=L^2(0,T;\mbbr)$ in the bilinear setting, $Y=H^1(0,T;\mbbr^N)$ and $Z=L^2(0,T;\mbbr^N) \times\mbbr^N$. Then, the assumptions in Lemma \ref{lem:AnalyticControlToStateMaps} on the ODE system and its linearization are satisfied for \eqref{eq:ODElinear} and \eqref{eq:ODElinearlinearized} in the linear setting, and for \eqref{eq:ODEbilinear} and \eqref{eq:ODElinearizedBilinear} in the bilinear setting.\footnote{Existence and uniqueness of all solutions $\ypmb,\delta\ypmb$ follow by Carathéodory's existence theorem (see, e.g., \cite[Theorem 54]{Sontag1998} and related propositions). For $\eps\in L^2(0,T;\mbbr^M)$ in the linear and $\epsilon\in L^2(0,T;\mbbr)$ in the bilinear setting, we obtain $\dot{\ypmb},\dot{\delta\ypmb}\in L^2(0,T;\mbbr^N)$ and thus $\ypmb,\delta\ypmb\in H^1(0,T;\mbbr^N)$.}
Notice that all solutions lie in $H^1(0,T;\mbbr^N)\Subset C(0,T;\mbbr^N)$ (see, e.g., \cite{Ciarlet}), which implies that the evolution map $S_T:H^1(0,T;\mbbr^N)\rightarrow\mbbr^N, \ypmb\mapsto\ypmb(T)$ is also linear and continuous.
Now, we can prove our main result.

\begin{theorem}[analysis for linear and bilinear systems]\label{thm:MainResult}
	Consider the linear setting \eqref{eq:ODElinear} or the bilinear setting \eqref{eq:ODEbilinear}.
	%For brevity, we
    Assume that the systems are sufficiently observable and controllable, i.e. fully observable and controllable in the linear case, and satisfying Assumption \ref{assum:Bilinear} in the bilinear case.
	If $\|A_{k+1}\|$ is sufficiently small, then the set of controls in $E_{ad}$ that satisfy \eqref{eq:AssumSplittingIneqOne}-\eqref{eq:AssumSplittingIneqTwo} in Assumption \ref{assum:LocalSplitting} is nonempty.
\end{theorem}
\begin{proof}
	For brevity, we denote $A_{\bet}:=A(\alcirc)+A^{(k)}(\bet^{k})$, $A_+:=A(\alcirc)+A_{k+1}$, $\ybet(\eps;t):=\ypmb(A_{\bet},\eps;t)$ and $\yplus(\eps;t):=\ypmb(A_+,\eps;t)$.
	
	We start with the linear setting \eqref{eq:ODElinear} from section \ref{sec:LinearReconstruction}.
	First, we derive observability and controllability properties for the systems $(A_+,B,C)$ and $(A_{\bet},B,C)$.
	Denote by $\sigma_k>0$ the smallest singular value of $\mathcal{O}_N(C,A(\alcirc))$.
	Let $k\in\{1,\ldots,K\}$ and $\bet^k\in\mbbr^k$ be the solution of \eqref{eq:FittingGeneralNonlin} for $\| A_{k+1}\|>0$ sufficiently small such that $\|\mathcal{O}_N(C,A(\alcirc))-\mathcal{O}_N(C,A_+)\|_2<\sigma_k$.
	From the proof of Theorem \ref{thm:FittingNonlinearToLinearized}, we obtain that also $\bet^k$ can be assumed to be sufficiently small such that $\|\mathcal{O}_N(C,A(\alcirc))-\mathcal{O}_N(C,A_{\bet})\|_2<\sigma_k$.
	Now, Lemma \ref{lem:rankdisturbedmatrix} guarantees that $\ranktxt(\mathcal{O}_N(C,A_+))=\ranktxt(\mathcal{O}_N(C,A_{\bet}))=N$.
	Using the same argument for the rank of the controllability matrices, we obtain that the systems $(A_+,B,C)$ and $(A_{\bet},B,C)$ are fully observable and controllable.
	
	Next, we consider the state of the difference $\zpmb(t)=\ypmb(A_+,\eps;t)-\ypmb(A_{\bet},\eps;t)$ with $\dot{\zpmb}=A_+\zpmb+(A_+-A_{\bet})\ypmb(A_{\bet},\eps;t)$.
	Since $A_+\neq A_{\bet}$, there exists $\vpmb\in\mbbr^N$ such that $(A_+-A_{\bet})\vpmb\neq0$.
	Recalling that $(A_{\bet},B)$ is controllable, we can find  $\eps_{t_1}$ for any $t_1\in(0,T]$ such that $\ybet(\eps_{t_1};)=\vpmb$ and therefore $(A_+-A_{\bet})\ybet(\eps_{t_1};t_1)\neq0$.
	We define
	\begin{equation*}
	\widetilde{\eps}(s):=\begin{cases}
	\eps_{t_1}(s),& \textnormal{for }0\leq s< t_1,\\
	\cpmb,&\textnormal{for }t_1\leq s\leq T,
	\end{cases}\label{eq:widetildeepslinear}
	\end{equation*}
	where $\cpmb\in\mbbr^N$ is to be chosen later.
	For $t>t_1$, we have 
	\begin{equation}\label{eq:zpmblinear}
	    \zpmb(t)=e^{(t-t_1)A_+}\zpmb(t_1)+\int_{t_1}^{t}e^{(t-s)A_+}(A_+-A_{\bet})\ybet(\widetilde{\eps};s)ds.
	\end{equation} 
	Multiplying \eqref{eq:zpmblinear} with $e^{-(t-t_1)A_+}$ from the left, we get
	\begin{equation*}
	    \widetilde{\zpmb}(t):=e^{-(t-t_1)A_+}\zpmb(t)=\zpmb(t_1)+\int_{t_1}^{t}e^{(t_1-s)A_+}(A_+-A_{\bet})\ybet(\widetilde{\eps};s)ds.
	\end{equation*}
	Notice that for $s>t_1$, the terms $e^{(t_1-s)A_+}$ and $\ybet(\widetilde{\eps};s)=e^{(s-t_1)A_{\bet}}\vpmb+\int_0^se^{(s-\tau)A_{\bet}}B\cpmb ds$ are continuous in $s$.
	Since exponential matrices are invertible (see, e.g., \cite[pag. 369, 5.6.P43]{horn_johnson_1991}) and $\zpmb(t_1)$ is independent of $t$, there exists a $t>t_1$ such that
	$\zpmb(t_1)+\int_{t_1}^{t}e^{(t_1-s)A_+}(A_+-A_{\bet})\ybet(\widetilde{\eps};s)ds\neq0$ and thus $\widetilde{\zpmb}(t)\neq0$.
	Using \eqref{eq:zpmblinear}, we obtain 
	\begin{equation}\label{eq:Czpmblin}
	    C\zpmb(t)=Ce^{(t-t_1)A_+}\widetilde{\zpmb}(t)=\sum_{j=0}^\infty \frac{(t-t_1)^j}{j!}CA_+^j\widetilde{\zpmb}(t).
	\end{equation}
	The observability of $(A_+,C)$ guarantees the existence of an integer $i\in\{0,\ldots,N-1\}$ such that $CA_+^i\widetilde{\zpmb}(t)\neq0$.
	We have $\frac{(t-t_1)^i}{i!}>0$ for $t>t_1$ and all terms of the sum in \eqref{eq:Czpmblin} converge to zero at different rates for different $j$. 
	Hence, there exists $t>t_1$ such that $C\zpmb(t)\neq0$.
	Since $t_1\in(0,T]$ was chosen arbitrarily, we obtain $C\zpmb(T)\neq0$ and thus $C\ybet(\widetilde{\eps};T)-C\yplus(\widetilde{\eps};T)\neq0$.
	
	Regarding the linearized system \eqref{eq:ODElinearlinearized}, we have already shown in Lemma \ref{lem:positivitymaxproblems} that there exists an $\eps\in E_{ad}$ such that $C\delta\ypmb_\circ(A^{(k)}(\bet^k),\eps;T)- C\delta\ypmb_\circ(A_{k+1},\eps;T)\neq0$.
	
	Finally, the maps $S,S_\delta: \Ltwo\rightarrow \mbbr^N$, $S(\eps):= C\ybet(\eps;T)-C\yplus(\eps;T)$, $S_\delta(\eps):= C\delta\ypmb_\circ(A^{(k)}(\bet^k),\eps;T)- C\delta\ypmb_\circ(A_{k+1},\eps;T)$ are analytic by Lemma \ref{lem:AnalyticControlToStateMaps}.
	Using Lemma \ref{lem:AnalyticFunctionsSplitting}, we obtain the existence of an $\eps\in E_{ad}$ such that $C\ypmb(A_{\bet},\eps;T)- C\ypmb(A_+,\eps;T)\neq0$ and $C\delta\ypmb_\circ(A^{(k)}(\bet^k),\eps;T)- C\delta\ypmb_\circ(A_{k+1},\eps;T)\neq0$.
	
	The proof for the bilinear setting \eqref{eq:ODEbilinear} from Section \ref{sec:BilinearReconstruction} is analogous to the one above and we omit it here for brevity.
	%%%%% SUPPL
	%For a detailed proof, we refer to the supplementary material \cite{supp_material}.
	%%%%% SUPPL
\end{proof}
\begin{remark}\label{rem:MainResult}
    Notice that we did not prove exactly Assumption \ref{assum:LocalSplitting} in Theorem \ref{thm:MainResult}, but only the existence of a general control $\eps\in E_{ad}$ that satisfies \eqref{eq:AssumSplittingIneqOne}-\eqref{eq:AssumSplittingIneqTwo}.
    However, this implies that any solution $\eps^{k+1}$ to \eqref{eq:SplittingGeneralNonlin} always satisfies \eqref{eq:AssumSplittingIneqOne}.
    Additionally, we recall from the proof of Theorem \ref{thm:MainResult} that the maps $S,S_\delta: L^2(0,T;\mbbr^M)\rightarrow \mbbr^N$, defined by $S(\eps):= \ypmb(A,\eps;T)$, $S_\delta(\eps):= \delta\ypmb_{\circ}(A,\eps;T)$ are nonzero.
    Thus, we obtain by Lemma \ref{lem:AnalyticFunctionsSplitting} that any neighborhood of $\eps^{k+1}$ can contain only isolated controls not satisfying \eqref{eq:AssumSplittingIneqTwo}, and infinitely many $\eps$ that do satisfy \eqref{eq:AssumSplittingIneqTwo}.
    Thus, it is rather unlucky to choose an $\eps^{k+1}$ not satisfying \eqref{eq:AssumSplittingIneqTwo}.
    On the other hand, one can also add inequality \eqref{eq:AssumSplittingIneqTwo} as a constraint to \eqref{eq:SplittingGeneralNonlin} to ensure that both inequalities are met by $\eps^{k+1}$.
\end{remark}
As a consequence of Theorems \ref{thm:MainTheorem}, \ref{thm:MainResult} and Remark \ref{rem:MainResult}, the controls generated by Algorithm \ref{algo:GeneralNonlin} for \eqref{eq:ODElinear} and \eqref{eq:ODEbilinear} make the GN matrix $\widehat{W}_\circ$, defined in \eqref{eq:Wwidehatcirc_ij}, PD under certain assumptions.
Thus, the results from Sections \ref{subsec:PosDefGaussNewtonLinear} and \ref{subsec:PosDefGaussNewtonBilinear} imply that GN for the reconstruction problems \eqref{eq:identificationlinear} and \eqref{eq:IdentificationProblemBilinear}, initialized with $\al_\circ$, converges to $\al_\star$.

\subsubsection{\comment{A general class of nonlinear systems}}\label{sub:sub:2}

\comment{
Now, we focus on a class of more general nonlinear systems and prove our second main result, which is analogous to Theorem \ref{thm:MainResult}.
These systems (of the form \eqref{eq:ODE}) are characterized by a function $f$ of the form $f(A,\ypmb,\eps)=g(\ypmb)+A\eps$, and thus have the form
\begin{equation}\label{eq:new_syst}
\dot{\ypmb}(t) = g(\ypmb(t))+A(\alstar)\eps(t),\quad t\in[0,T],
\end{equation}
where $\ypmb(t)\in\mbbr^N,N\in\mbbn$.
We make the following hypotheses.
\begin{assumption}[on the system \eqref{eq:new_syst}]\label{ass:new}
    The initial state is $\ypmb(0)=0$.
	The function $g$ is in $C^1$  and satisfies $g(0)~=~0$.\footnote{\comment{Notice that the first two hypotheses imply that the uncontrolled system ($\eps = 0$) system has an equilibrium point at zero.}} Its Lipschitz constant is denoted by $L$.
    %The control space $E_{ad}\subset L^2(0,T;\mbbr^M)$ is bounded, closed, convex and contains $0$ as an interior point.
	The approximation of $\alstar$ is given by $\al_\circ=0$.
	The state $\ypmb(T)$ can be fully observed, meaning that the observer matrix is the identity: $C=I\in\mathbb{R}^{N\times N}$.
\end{assumption}
}

\comment{
	Assumption~\ref{ass:new} implies that the linearized  equation around $\alcirc=0$ is given by 
	\begin{align*}
	\dot{\delta\ypmb}_\circ(t)&=g'(\ypmb_\circ(t))\delta\ypmb_\circ(t) + A(\al)\eps(t),\quad \delta\ypmb(0)=0,\\
	\dot{\ypmb}_\circ(t)&= g(\ypmb_\circ(t)),\quad \ypmb_\circ(0)=0,
	\end{align*}
    where we used that $A(\al_\circ=0)=0$.
	Using again Assumption~\ref{ass:new}, and in particular that $\ypmb_\circ(0)=0$ and $g(0)=0$, we obtain $\ypmb_\circ(t)=0$ for all $t\in[0,T]$ and therefore
	\begin{equation*}
	\dot{\delta\ypmb}_\circ(t)=g'(0)\delta\ypmb_\circ(t) + A(\al)\eps(t).
	\end{equation*}
 }
 \comment{
	Now, we show that there exists a control satisfying inequality \eqref{eq:AssumSplittingIneqOne} in Assumption~\ref{assum:LocalSplitting}.
	We then introduce the notation
	$\widetilde{A}:=A^{(k)}(\bet^k)-A_{k+1}$, $\ypmb_{\bet}(t):=\ypmb(A^{(k)}(\bet^k),\eps^{k+1};t)$, 
	$\ypmb_{k+1}(t):=\ypmb(A_{k+1},\eps^{k+1};t)$, 
	and $\Delta\ypmb(t):=\ypmb_{\bet}(t)-\ypmb_{k+1}(t)$.
	Then, $\Delta\ypmb(t)$ satisfies
	\begin{equation}\label{eq:dotDeltaypmb}
	\dot{\Delta\ypmb}(t)=g(\ypmb_{\bet}(t))-g(\ypmb_{k+1}(t))+\widetilde{A}\eps^{k+1}(t), \quad \Delta\ypmb(0) = 0.
	\end{equation}
	By multiplying both sides of \eqref{eq:dotDeltaypmb} with $\widetilde{A}\eps^{k+1}(t)$,
    %we obtain
	%\begin{equation*}
	%\langle \widetilde{A}\eps^{k+1}(t),\dot{\Delta\ypmb}(t)\rangle=\langle \widetilde{A}\eps^{k+1}(t),g(\ypmb_{\bet}(t))-g(\ypmb_{k+1}(t))\rangle+\|\widetilde{A}\eps^{k+1}(t)\|_2^2.
	%\end{equation*}
	%Since 
    and using that $g$ is Lipschitz continuous with Lipschitz constant $L$, we have
	\begin{align*}
	\|\widetilde{A}\eps^{k+1}(t)\|_2^2&= \langle \widetilde{A}\eps^{k+1}(t),\dot{\Delta\ypmb}(t)\rangle+\langle \widetilde{A}\eps^{k+1}(t),g(\ypmb_{k+1}(t))-g(\ypmb_{\bet}(t))\rangle\\
    &{\leq}\|\widetilde{A}\eps^{k+1}(t)\|_2\Big(\|\dot{\Delta\ypmb}(t)\|_2+L\|\Delta\ypmb(t)\|_2 \Big),
	\end{align*}
    where we used Cauchy-Schwarz and triangle inequalities.
	Thus, we obtain
	\begin{equation}\label{eq:estim_new}
	\|\widetilde{A}\eps^{k+1}(t)\|_2\leq\|\dot{\Delta\ypmb}(t)\|_2+L\|\Delta\ypmb(t)\|_2.
	\end{equation}
}
\comment{
	Since $A^{(k)}(\bet^k)$ and $A_{k+1}$ are linearly independent, we have that $\widetilde{A}\neq0$. Thus, there exists $\vpmb\in\mbbr^M$ such that $\widetilde{A}\vpmb\neq0$.
	Now, we define the control %$\eps_{\vpmb,\tau}$ as
	%\begin{equation*}
	$\eps_{\vpmb,\tau}(t):=
	\begin{cases*}
	0,\quad t\in[0,\tau),\\\vpmb,\quad t\in[\tau,T],
	\end{cases*}$
	%\end{equation*}
	for an arbitrary $\tau\in(0,T)$.
	Hence, using \eqref{eq:estim_new}, for all $t>\tau$ it holds that
	\begin{equation*}
	\|\dot{\Delta\ypmb}(t)\|_2+L\|\Delta\ypmb(t)\|_2 \geq \|\widetilde{A}\eps_{\vpmb,\tau}(t)\|_2>0.
	\end{equation*}
 }
 \comment{
    Recalling that $C=I$ by Assumption~\ref{ass:new},
	if $\|\Delta\ypmb(T)\|_2>0$, the first inequality \eqref{eq:AssumSplittingIneqOne} is satisfied and we could move to inequality \eqref{eq:AssumSplittingIneqTwo}.
	Thus, assume that $\|\Delta\ypmb(T)\|_2=0$. Then $\|\dot{\Delta\ypmb}(T)\|_2\neq0$ and therefore $\|\Delta\ypmb(T-\eta)\|_2\neq0$ for some $\eta>0$. Notice that $\eta$ is independent of $\tau$. 
	Since $\Delta\ypmb(t)=0$ for $t\in[0,\tau)$, we can simply shift $\tau$ by $\eta$, namely considering $\eps_{\vpmb,\tau+\eta}$, to obtain $\|\Delta\ypmb(T)\|_2>0$.
	In conclusion, we have found a control, namely $\eps^{k+1}=\eps_{\vpmb,\tau+\eta}$, satisfying \eqref{eq:AssumSplittingIneqOne}.
	%Thus, it must also hold for the maximizer, i.e. any solution $\eps^{k+1}$ to (7.3).
}

\comment{Now, we show that $\eps^{k+1}=\eps_{\vpmb,\tau+\eta}$ satisfies \eqref{eq:AssumSplittingIneqTwo} in Assumption~\ref{assum:LocalSplitting} as well.
	To do so, let 
    $\Delta\ypmb_\delta(t):=\delta\ypmb_\circ(A^{(k)}(\bet^k),\eps^{k+1};t)-\delta\ypmb_\circ(A_{k+1},\eps^{k+1};t)$,
	which clearly satisfies
	\begin{equation*}
	\dot{\Delta\ypmb}_\delta(t)=g'(0)\Delta\ypmb_\delta(t)+\widetilde{A}\eps^{k+1}(t),\quad \Delta\ypmb_\delta(0)=0.
	\end{equation*}
    %Thus $t\mapsto \Delta\ypmb_\delta(t)$ is continuous.
    Now, since $\widetilde{A}\eps^{k+1}(t)=\widetilde{A}\eps_{\vpmb,\tau+\eta}(t)\neq0$ for $t\in[\tau+\eta,T]$, thus $\dot{\Delta\ypmb}_\delta(\tau+\eta)\neq0$ 
    Hence, since $t\mapsto \Delta\ypmb_\delta(t)$ is continuous, there exists a $\widetilde{\eta}>0$ such that $\tau+\widetilde{\eta} \in (\tau+\eta,T]$ and
    $\Delta\ypmb_\delta(t) \neq 0$ for all $t\in (\tau+\eta,\tau+\widetilde{\eta}]$.
    If $\Delta\ypmb_\delta(T) \neq 0$ we proved our claim. Thus, we assume that
    $\tau+\widetilde{\eta}<T$ and $\Delta\ypmb_\delta(T)=0$. In this case, we can use the fact that $\tau$ is arbitrary and choose it to shift the interval $(\tau+\eta,\tau+\widetilde{\eta}]$ in order to get that $T\in (\tau+\eta,\tau+\widetilde{\eta}]$. This choice of $\tau$ implies that the control $\eps^{k+1}=\eps_{\vpmb,\tau+\eta}$ satisfies both inequalities \eqref{eq:AssumSplittingIneqOne} and \eqref{eq:AssumSplittingIneqTwo} in Assumption~\ref{assum:LocalSplitting}. 
 }

 \comment{
We can summarize our findings in the next theorem, which is the analog of Theorem~\ref{thm:MainResult} for \eqref{eq:new_syst}. 
\begin{theorem}[\comment{analysis for systems of the form \eqref{eq:new_syst}}]\label{thm:MainResult2}
	Consider the system \eqref{eq:new_syst} and let Assumption~\ref{ass:new} hold.
	Then the set of controls in $E_{ad}$ that satisfy both \eqref{eq:AssumSplittingIneqOne} and \eqref{eq:AssumSplittingIneqTwo} in Assumption \ref{assum:LocalSplitting} is nonempty.
\end{theorem}
\begin{remark}
    Theorem~\ref{thm:MainResult2} is an alternative proof of global convergence for the linear-quadratic setting considered in \cite[Section 4]{siampaper}, where the unknown is the control matrix.
    However, it does not cover the linear setting discussed in section \ref{sec:LinearReconstruction}, where the unknown is the drift matrix.
\end{remark}
\begin{remark}
    As in Remark \ref{rem:MainResult}, Theorem \ref{thm:MainResult2} does not exactly prove Assumption \ref{assum:LocalSplitting}, but rather that there exist controls that simultaneously satisfy \eqref{eq:AssumSplittingIneqOne} and \eqref{eq:AssumSplittingIneqTwo}.
    On the other hand, the map $\eps\mapsto\Delta\ypmb_\delta(T)=\delta\ypmb_\circ(A^{(k)}(\bet^k),\eps;T)-\delta\ypmb_\circ(A_{k+1},\eps;T)$ is analytic by Lemma~\ref{lem:AnalyticControlToStateMaps} (under certain regularity of $f$), and the map $t\mapsto\Delta\ypmb_\delta(t)$ is (at least) continuous.
    In this case, the control functions $\eps$ that violate simultaneously \eqref{eq:AssumSplittingIneqOne} and \eqref{eq:AssumSplittingIneqTwo} are isolated points.
    Therefore, it is theoretically possible but also very unlikely that the solutions $\eps^{k+1}$ to the nonlinear splitting step problem \eqref{eq:SplittingGeneralNonlin} do not satisfy inequality \eqref{eq:AssumSplittingIneqTwo} in Assumption \ref{assum:LocalSplitting}.
\end{remark}
}

\subsection{Optimized GR Algorithm}\label{sec:Improvements}
The analyses discussed in the previous sections are based on certain hypotheses of observability and controllability of the dynamical system. 
However, as shown already in \cite{siampaper},
%%%%% SUPPL
%and also discussed in the supplementary material \cite{supp_material},
%%%%% SUPPL
if they are not satisfied, the choice of the elements of $\mathcal{A}$ becomes very relevant and can strongly affect the online reconstruction process. 
For this reason, a modified GR algorithm  %. This is 
called Optimized GR (OGR) has been introduced in~\cite{siampaper} to identify important basis elements by solving in each iteration the fitting and splitting step problems (in parallel) for all remaining basis elements. %, and not just the next one.
This also allows us to initialize the algorithm with a larger number of elements $(A_j)_{j=1}^K$, i.e., $K>N^2$.
Even though some of the matrices $A_j$ will inevitably be linearly dependent if $K>N^2$, the OGR algorithm manipulates them to construct a new subset of linearly independent ones. 
In the spirit of the previous analysis, we add a new feature to the original OGR algorithm.
At iteration $k$, after all fitting step problems have been solved, we check whether there exists $\ell\in\{k+1,\ldots,K\}$ for which the optimal cost function value is different from zero (i.e. larger than a tolerance $\textrm{tol}_2$).
If this is the case, then there exists a control $\eps^m$, $m\in\{1,\ldots,k\}$, that already satisfies $\norm{C\ypmb(A^{(k)}(\bet^{\ell}),\eps^m;T)-C\ypmb(A_{\ell},\eps^m;T)}_2^2>\textrm{tol}_2$ for at least one index $\ell_{k+1}\in\{k+1,\ldots,K\}$ (see Step 8 in Algorithm~\ref{alg:OGRlinear}). 
Hence, we can add the element $A_{\ell_{k+1}}$ to the already selected ones without computing a new control.
This new improvement can also be motivated with the matrix formulation we used for our analysis. If $\ranktxt(\widehat{W}_\circ^{(k)})=r>k$, one can appropriately permute rows and columns of $\widehat{W}_\circ^{(k)}$ such that $[\widehat{W}_\circ^{(k)}]_{[1:r,1:r]}$ has rank $r$ and is thus PD.
The rank of $\widehat{W}_\circ^{(k)}=\sum_{m=1}^{k}W_\circ(\eps^m)$ is bounded by $kP$ (recall $C \in \mathbb{R}^{P \times N}$).
This can be seen by writing $W_\circ(\epsilon^m)$, as defined in \eqref{eq:Wwidehatcirc_ij}, as $W_\circ(\epsilon^m)=\delta Y_\circ^\top C^\top C\delta Y_\circ$, where $\delta Y_\circ:=\begin{bmatrix}
\delta\ypmb_\circ(A_1,\eps^m;T),\cdots,\delta\ypmb_\circ(A_K,\eps^m;T)
\end{bmatrix}$.
Hence, $\ranktxt(W_\circ(\epsilon^m))\leq \ranktxt(C)\leq P$, and thus $\ranktxt(\widehat{W}_\circ^{(k)})\leq kP$.

OGR is stated in Algorithm \ref{alg:OGRlinear}, where the new feature corresponds to the steps 8-9.
Algorithm \ref{alg:OGRlinear} can also be formulated for the linearized setting of the previous sections by replacing $\ypmb$ with its linearization $\delta\ypmb_\circ$.
We call OLGR the OGR algorithm for the linearized system.
%Note that, in this case, we have $C\ypmb_\circ(0,0;T)=0$.

\begin{algorithm}[t]
	\caption{Optimized Greedy Reconstruction (OGR) Algorithm}
	\begin{small}
		\begin{algorithmic}[1]\label{alg:OGRlinear}
			\REQUIRE A set of $K$ matrices $\mathcal{A}=\{A_1,\ldots,A_K\}$
			and two tolerances $\textrm{tol}_1>0$ and $\textrm{tol}_2>0$.
			\STATE Set $\eps^0=0$ and compute $\eps^1$ and the index $\ell_1$ by solving the initialization problem
			\begin{equation*}%\label{eq:bilinear_initialization_regularized}
			\max_{\ell\in\{1,\ldots,K \}}\max_{\eps\in E_{ad}} \norm{C\ypmb(0,\eps;T)-C\ypmb(A_{\ell},\eps;T)}_2^2.
			\end{equation*}
			%\IF {$\norm{C\ypmb\klam{A_j,\eps;T}}_2^2< {\rm tol}_1$}
			%\STATE {\bf stop} and display "Error: all basis elements have no observable effect."
			%\ENDIF
			\STATE Swap $A_1$ and $A_{\ell_1}$ in $\mathcal{A}$, and set $k=1$ and $A^{(0)}(\bet^{\ell_1})=0$.
			\WHILE{ $k\leq K-1$ and \begin{small}$\norm{C\ypmb(A^{(k-1)}(\bet^{\ell_k}),\eps^k;T)-C\ypmb(A_{k},\eps^k;T)}_2^2 > {\rm tol}_1$\end{small} }
			\FOR{$\ell=k+1,\ldots,K$}
			\STATE Orthogonalize all elements $(A_{k+1},\ldots,A_K)$ with respect to $(A_1,\ldots,A_k)$, remove any that are linearly dependent and update $K$ accordingly.
			\STATE \underline{Fitting step}: Find $(\bet^{\ell}_j)_{j=1,\dots,k}$ that solve the problem
			\begin{equation*}%\label{eq: greedy fitting step Ham}
			\min_{\bet\in\mbbr^{k}} \sum_{m=1}^{k}\norm{C\ypmb(A^{(k)}(\bet),\eps^m;T)-C\ypmb(A_{\ell},\eps^m;T)}_2^2,
			\end{equation*}
			and set $f_\ell = \sum_{m=1}^{k} \begin{small}\norm{C\ypmb(A^{(k)}(\bet^\ell),\eps^m;T)-C\ypmb(A_{\ell},\eps^m;T)}_2^2\end{small}$.
			\ENDFOR
			\IF {$\max_{\ell=k+1,\dots,K} f_\ell > {\rm tol}_2$ }
			\STATE Set $\ell_{k+1} = \argmax_{\ell=k+1,\dots,K} f_\ell$.
			\ELSE
			\STATE \underline{Extended splitting step}: Find $\eps^{k+1}$ and $\ell_{k+1}$ that solve the problem
			\begin{equation*}%\label{eq: greedy discriminatory step Ham}
			\max_{\ell\in\{k+1,\ldots,K \}}\max_{\eps\in E_{ad}}\norm{C\ypmb(A^{(k)}(\bet^{\ell}),\eps;T)-C\ypmb(A_{\ell},\eps;T)}_2^2.
			\end{equation*}
			%\IF {$\norm{C\ypmb(A^{(k)}(\bet^{\ell}),\eps;T)-C\ypmb(A_{\ell},\eps;T)}_2^2< {\rm tol}_1$}
			%\STATE {\bf stop} and return $\mathcal{A}$ and the computed $(\eps^m)_{m=0}^k$.
			%\ENDIF
			\ENDIF
			\STATE Swap $A_{k+1}$ and $A_{\ell_{k+1}}$ in $\mathcal{A}$, and set $k = k+1$.
			\ENDWHILE
		\end{algorithmic}
	\end{small}
\end{algorithm}
\setlength{\textfloatsep}{8pt}

\subsection{\comment{Global convergence in a specific case}}\label{sec:global}
\comment{In the above sections, we discussed the performance of the GR algorithm on the local convergence of GN for the online identification problem. 
In the context of inverse problems for nonlinear ODEs, global convergence can generally not be expected, independently of the method used to generate control functions or data.
A sufficient condition for the global convergence of GN applied to such problems is strict convexity of the online reconstruction problem \eqref{eq:identificationproblem}.
Thus, if GR can produce controls making the cost of \eqref{eq:identificationproblem} strictly convex, then GN (or other optimization solvers) exhibits global convergence.
One example, where strict convexity can be proven under certain assumptions, is given by the linear-quadratic setting discussed in \cite[Section 4]{siampaper}.
In what follows, we discuss global convergence (convexity) for a specific case of a bilinear system.}

\comment{
	The goal is to reconstruct a matrix $A_\star=\al_{\star,1}A_1+\al_{\star,2}A_2\in\mbbr^{2\times2}$ of the system
	\begin{equation}\label{eq:bilinearexample}
	\dot{\ypmb}(t)=\epsilon(t)A_\star\ypmb(t),\quad t\in[0,T],\quad \ypmb(0)=\ypmb^0,
	\end{equation}
	where $T=1$. Assume that
	$\ypmb^0=
    \begin{small}    
    \begin{bmatrix}
	1\\0
	\end{bmatrix}
    \end{small}
    $,
    $\alstar =\alcirc= 
    \begin{small} 
    \begin{bmatrix}
	0\\0
	\end{bmatrix}
    \end{small}$,
    $A_1=\begin{small} 
    \begin{bmatrix}
	1&0\\0&-1
	\end{bmatrix}
    \end{small}$,
    $A_2=\begin{small} 
    \begin{bmatrix}
	0&1\\1&0
	\end{bmatrix}
    \end{small}$.
    We also assume that we can observe the full state $\ypmb(T)$, i.e., $C=I \in \mbbr^{2\times2}$.}
    \comment{
    We consider that $E_{ad}$ is a set of constant controls that are bounded in absolute value by 1, i.e., $\epsilon(t)\equiv \epsilon\in U_{ad}=[-1,1]$.
	Thus, the solution to \eqref{eq:bilinearexample} can be written explicitly as $\ypmb(\al,u;t)=e^{t\epsilon A(\al)}\ypmb^0$.
	Hence, the final reconstruction problem \eqref{eq:identificationproblem} reads as
	\begin{equation}\label{eq:globalreconstructionproblem}
	\min_{\al\in\mbbr^2}J(\al):=\sum_{m=1}^{2}\|e^{(\epsilon_m(\al_1A_1+\al_2A_2))}\ypmb^0-e^{\epsilon_m A_\star}\ypmb^0\|_2^2.
	\end{equation}
	Since $A_\star$ is zero, we have $e^{\epsilon A_\star}\ypmb^0=\ypmb^0$.
	Using the expressions for $A_1$ and $A_2$ and the result \cite[Cor 2.3]{MatrixExponentials}, one gets 
	$e^{\epsilon(\al_1A_1+\al_2A_2)}\ypmb^0
    =
    \frac{1}{\|\al\|_2}
    \begin{small}
    \begin{bmatrix}
	\|\al\|_2\cosh(\|\al\|_2\epsilon)+\al_1\sinh(\|\al\|_2\epsilon)\\
	\al_2\sinh(\|\al\|_2\epsilon)
	\end{bmatrix}\end{small}$.
    Thus, by a direct calculation, we can write the cost function $J$ in \eqref{eq:globalreconstructionproblem} as
	\begin{align*}
	J(\al)	&=\sum_{m=1}^2\Big(\cosh(\|\al\|_2\epsilon_m)^2+\sinh(\|\al\|_2\epsilon_m)^2-2\cosh(\|\al\|_2\epsilon_m)\\
	&\quad+\frac{2\al_1\sinh(\|\al\|_2\epsilon_m)}{\|\al\|_2}(\cosh(\|\al\|_2\epsilon_m)-1)+1\Big)
	\end{align*}
	for $\al\neq0$ and $J(0)=0$.
	Assume now that $\epsilon_1=-1$ and $\epsilon_2=1$.
	Changing the variable to $x=\|\al\|_2$, and using that $\sinh(-x)=-\sinh(x)$ and $\cosh(-x)=\cosh(x)$, we get
	\begin{equation*}
	J(\al) = 2\cosh(\|\al\|_2)^2+2\sinh(\|\al\|_2)^2-4\cosh(\|\al\|_2)+2 =: \widetilde{J}(x).
	\end{equation*}
	Thus, $\widetilde{J}''(x)= 8\cosh(x)^2+8\sinh(x)^2-4\cosh(x)$.
	Since $\cosh(x)\geq1$ and $\sinh(x)\geq0$ for $x\geq0$, it holds that $\widetilde{J}''(x)>0$ for $x\geq0$.
	Therefore, $J$ is strictly convex.
	Thus, a sufficient condition for global convergence of GN is that our proposed methods select the two controls $\epsilon_1=-1$ and $\epsilon_2=1$ (or $\epsilon_1=1$ and $\epsilon_2=-1$).
    Notice that, if only one control is used among $1$ and $-1$, then $J$ is not convex.
    }

    \comment{
	Let us now study the GR algorithm.
    The initialization problem is given by
	\begin{equation*}\label{eq:GlobalInitialization}
	\max_{\epsilon\in[-1,1]}J_I(u):=\|\ypmb^0-e^{\epsilon A_1}\ypmb^0\|_2^2,
	\end{equation*}
	where, by a direct calculation, we get $J_I(\epsilon) =  (\exp(\epsilon)-1)^2$.
    One can show that $J_I(\epsilon)$ attains its unique global maximum at $\epsilon=1$. % with a function value of $J(1)\approx2.95$.
	Thus, assuming that the splitting step solver converges to the unique global maximizer, GR will choose $\epsilon_1=1$.
	Let us now consider the fitting step problem
	\begin{equation*}
	\min_{\alpha\in\mbbr}J_F(\alpha):=\|e^{\epsilon_1\alpha A_1}\ypmb^0-e^{\epsilon_1A_2}\ypmb^0\|_2^2.
	\end{equation*}
	Direct calculations lead to
	%\begin{equation*}
	$J_F(\alpha)=e^{2\alpha}-2\cosh(1)e^\alpha+\cosh(1)^2+\sinh(1)^2$ and
	$J_F'(\alpha)=2e^{\alpha}(e^\alpha-\cosh(1))$.
	%\end{equation*}
	Clearly, $J_F'(\log(\cosh(1)))=0$ and $J_F'(\alpha)$ is negative for $\alpha<\log(\cosh(1))$ and positive for $\alpha>\log(\cosh(1))$.
	Since it is continuous in $0$, $J_F$ has a global minimum in $\alpha_1:=\log(\cosh(1))$.
	Now, consider the splitting step problem
	\begin{equation}\label{eq:maxeps}
	\max_{\epsilon\in[-1,1]}J_S(\epsilon):=\|e^{\epsilon \alpha_1A_1}\ypmb^0-e^{\epsilon A_2}\ypmb^0\|_2^2.
	\end{equation}
    Proceeding as before, one can get that \eqref{eq:maxeps} has a unique global maximizer given by $\epsilon=-1$. % with a function value of $J(1)\approx2.18$.
	Thus, assuming again that the splitting step solver converges to the unique global maximizer, GR will choose $\epsilon_2=1$.}

    \comment{
	In conclusion, GR chooses $\epsilon_1=1$ and $\epsilon_2=-1$, making the final reconstruction problem \eqref{eq:globalreconstructionproblem} strictly convex and hence leading to global convergence of GN.}

    \comment{
    The above example is very instructive. On the one hand, as we are going to see, it can be used to reveal some properties (and weaknesses) of GR. 
    On the other hand, it is straightforward  to carry out similar calculations to compare GR, LGR and OGR.
    These calculations (omitted here for brevity) lead to the following remarks.}

    \comment{
        The above results about GR for the specific case \eqref{eq:bilinearexample} depend heavily on the order of the basis elements $A_1$ and $A_2$.
        If one repeats the calculation with the order of $A_1$ and $A_2$ reversed,
        then the initialization problem has two global maximizers, located at $\epsilon=\pm 1$. Independently of which of these is chosen, the fitting step has two global minimizers. Each of them leads to a splitting step problem having two global maximizers located again at $\epsilon=\pm1$, independently of the chosen value in the initialization step.
        Therefore, even if one assumes that the optimizations solvers used for intialization, fitting and splitting steps converge to global (maximum/minimum) points, it can happen that the control computed at the splitting step is equal to the one obtained at the initialization step, making the cost of the final identification problem not convex.
        The importance of the ordering of the elements in $\mathcal{A}$ was already discussed in detail in \cite{siampaper} for the reconstruction of the control matrix $B$ in case of linear systems.
        This was exactly the reason for designing OGR. 
        In fact, if one repeats the above calculation for OGR, the two controls $\epsilon=1$ and $\epsilon=-1$ are always obtained, independently of the ordering of $A_1$ and $A_2$.
        Finally, for LGR one can show that, independently of the ordering of $A_1$ and $A_2$, it can happen that the controls computed at the initialization and splitting steps are equal, leading to a final identification problem which is not convex.
        This is also due to the fact that the linearization process does not allow to fully capture the dynamics of the system.
        Such behavior is also apparent in the numerical experiments performed in the next section.
        }

\section{Numerical experiments}\label{sec:numerics}
In this section, efficiency and robustness of the GR and OGR algorithms are studied by direct numerical experiments. First, we consider the reconstruction of a drift matrix in Section \ref{subsec:NumLinear}. Second, we focus on the reconstruction of a bilinear dipole momentum operator as in Section \ref{subsec:NumBilinear}.
All optimization problems in the GR algorithms are solved by a BFGS descent-direction method. The online identification problem is solved by GN.

\subsection{Reconstruction of drift matrices}\label{subsec:NumLinear}
We consider system \eqref{eq:ODElinear} with (full rank) randomly generated matrices $\Astar,B,C\in\mbbr^{3\times3}$.
The final time is $T=1$ and the initial value is $\ypmb^0=[0,0,0]^\top$.
First, we study the algorithms for system \eqref{eq:ODElinearlinearized}. 
This is obtained by linearizing \eqref{eq:ODElinear} around two different $A_\circ$, which are randomly chosen approximations to $\Astar$, one with $1\%$ and the other with $10\%$ relative error, meaning that, e.g., $\frac{\|\Astar-A_\circ\|_F}{\|\Astar\|_F}=0.01$ for the one with $1\%$ error, where $\|\cdot\|_F$ is the Frobenius norm.
The LGR Algorithm \ref{algo:General} is run for two different choices for the basis $\mathcal{A}$: the canonical basis of $\mbbr^{3\times3}$ and a basis consisting of 9 randomly generated (linearly independent) $3\times3$ matrices.
LGR is also compared with the OLGR Algorithm \ref{alg:OGRlinear}, which is run with a set of 18 matrices, namely, the 9 canonical basis elements and the 9 random matrices.
The controls generated by the respective algorithms are then used to reconstruct the matrix $\Astar$ by solving the online least-squares problem \eqref{eq:identificationproblem} with GN.
To test the robustness of the control functions, we consider a nine-dimensional sphere centered in the global minimum $\Astar$ and with given relative radius $r$, and repeat the minimization for 1000 initialization vectors randomly chosen on this sphere.
We then count the percentage of times that GN converges to the global solution $\Astar=A(\alstar)$ up to a tolerance of $Tol=0.005$ (half of the smallest considered radius), meaning that $\frac{\|\Astar-A(\al_{comp})\|_F}{\|\Astar\|_F}\leq Tol$, where $\al_{comp}$ denotes the solution computed by GN.
Repeating this experiment for different radii of the sphere, we obtain the results reported in the two panels on the left in Figure \ref{fig:LinearSysLinearizedSetting}.
\begin{figure}[t]
\centering
\begin{subfigure}{.5\textwidth}
  \centering
  \includegraphics[width=0.9\textwidth]{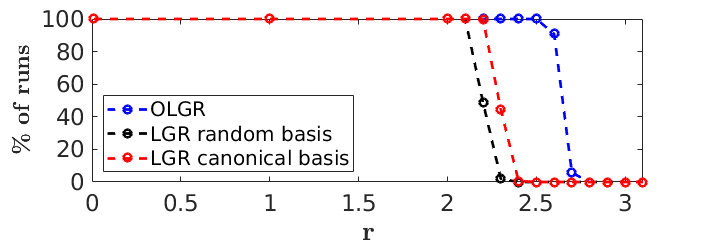}
  \includegraphics[width=0.9\textwidth]{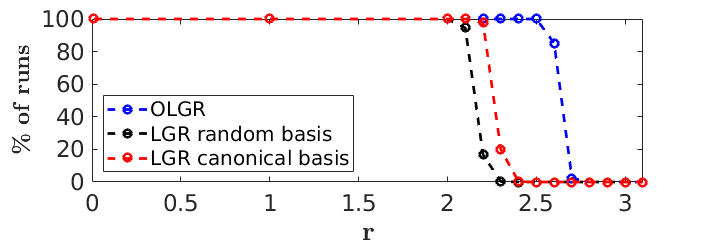}
  %\caption{A subfigure}
  \label{fig:sub1}
\end{subfigure}%
\begin{subfigure}{.5\textwidth}
  \centering
  \includegraphics[width=0.9\textwidth]{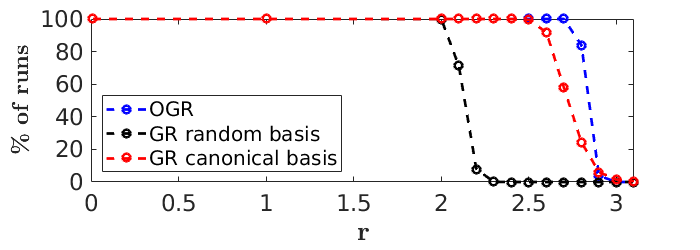}
  \includegraphics[width=0.9\textwidth]{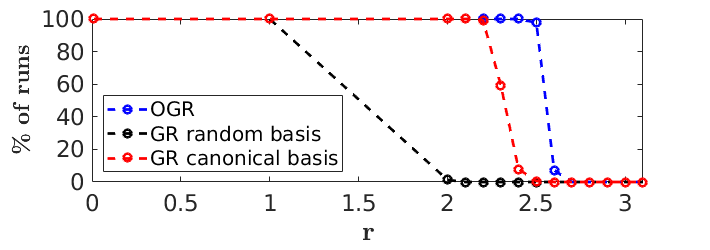}
  %\caption{A subfigure}
  \label{fig:sub2}
\end{subfigure}
\vspace*{-4mm}
\caption{Percentage of runs that converged to $A_\star$ initialized by randomly chosen vectors on a %nine-dimensional 
sphere with radius $r$, for controls generated by LGR and OLGR for $1\%$ (top left) and $10\%$ (bottom left) relative error between $A_\star$ and $A_\circ$, and GR and OGR Algorithm \ref{algo:GeneralNonlin} with (bottom right) and without the shift by $A_\circ$ (top right).}
	\label{fig:LinearSysLinearizedSetting}
\end{figure}
All control sets make GN capable of reliably reconstructing the global minimum $\Astar$ up to a relative radius $r=2$, which corresponds to a relative error of $200\%$.
This demonstrates that the choice of the basis is not crucial for fully observable and controllable systems.
However, OLGR is able to reduce the number of controls down to $3$ and still outperforms any set of $9$ controls generated by LGR, while staying reliable up to a relative error of $250\%$.
Thus, OLGR is able to compute a better basis %, thereby optimizing the performance, 
and to omit unnecessary controls.

Next, we repeat the same experiments for the GR Algorithm \ref{algo:GeneralNonlin}.
However, we replace the case for the approximation $A_\circ$ with a relative error of $1\%$ by $A_\circ=0$.
This effectively removes the shift and makes the algorithm independent of the choice of $A_\circ$, which is the version of the algorithm that was also considered in \cite{madaysalomon,siampaper}.
We obtain the results shown in the two panels on the right in Figure \ref{fig:LinearSysLinearizedSetting}.
The performance of the control sets is similar to the ones for the linearized system, with an increase in performance for the GR algorithm with the canonical basis, without the shift by $A_\circ$, and a decrease in performance for the GR algorithm with the random basis and an $A_\circ$ that has a $10\%$ relative error with respect to $A_\star$.
As in the linearized setting, OGR is able to reduce the number of controls down to $3$ and still outperforms any set of $9$ controls generated by LGR.
%The controls generated by the GR algorithm for the canonical basis without the shift by $A_\circ$ experiences a greater increase in performance, almost matching the one of the OGR algorithm.

\subsection{Bilinear reconstruction problem}\label{subsec:NumBilinear}
Similar to \cite{madaysalomon} and \cite{siampaper}, we consider a Schr\"odinger-type equation, written as a real system as in \eqref{eq:ODESchroedingerReal}.
We also use matrices $H$ and $\mustar$ similar to the ones used in \cite{siampaper}, namely
\begin{equation*}
H = H_R =
\begin{footnotesize}
\begin{bmatrix}
4&0&0\\
0&8&0\\
0&0&16
\end{bmatrix}
\end{footnotesize}, \;
\mustar = 
\begin{footnotesize}
\begin{bmatrix}
-0.3243  &  -3.4790+0.7359i  &  -0.5338+1.9254i \\
-3.4790-0.7359i  &  -3.8342  &  -1.1697+2.0256i \\
-0.5338-1.9254i  &  -1.1697-2.0256i & 1.0551\\
\end{bmatrix}.
\end{footnotesize}
\end{equation*}
\vspace{-2mm}

\noindent
The final time is $T=10\pi$ and the initial state is $\psipmb_0=[1,0,0]^\top$.
The observer matrix is $C=[\psipmb_1,i\psipmb_1]$, which means that the final state is measured against the fixed state $\psipmb_1=\frac{1}{\sqrt{3}}[1,1,1]^\top$.
For $E_{ad}$ we use the  box constraints $|\epsilon(t)|\leq1$.
Again, we consider two bases, each consisting of $9$ elements: the canonical and a random one for the space of Hermitian matrices in $\mbbc^{3\times 3}$.
We then perform the same experiments as in Section \ref{subsec:NumLinear}.
The results are reported in Figure \ref{fig:BilinearSysLinearizedSetting}.
\begin{figure}[t]
\centering
\begin{subfigure}{.5\textwidth}
  \centering
  \includegraphics[width=0.9\textwidth]{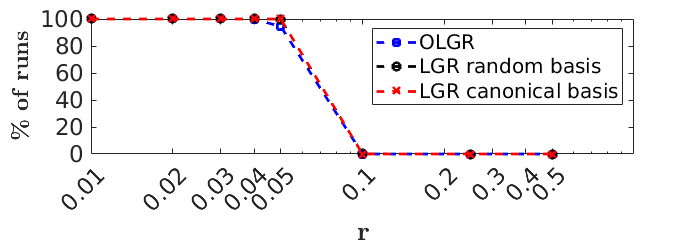}\\
  \includegraphics[width=0.9\textwidth]{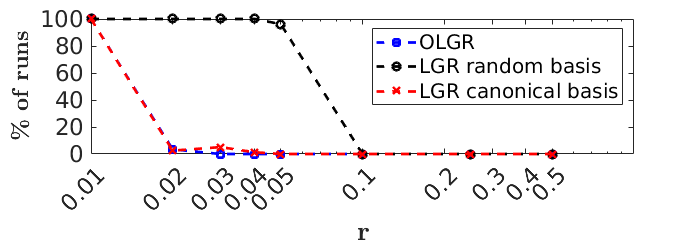}
  %\caption{A subfigure}
  \label{fig:BilinearSysLinearizedSetting1}
\end{subfigure}%
\begin{subfigure}{.5\textwidth}
  \centering
  \includegraphics[width=0.9\textwidth]{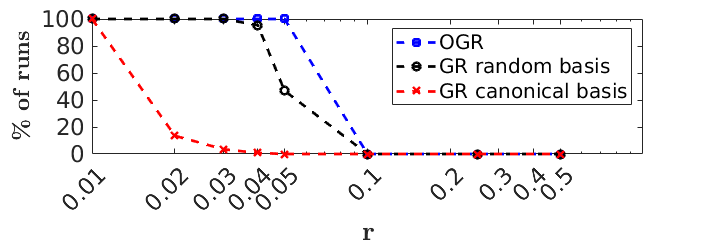}\\
  \includegraphics[width=0.9\textwidth]{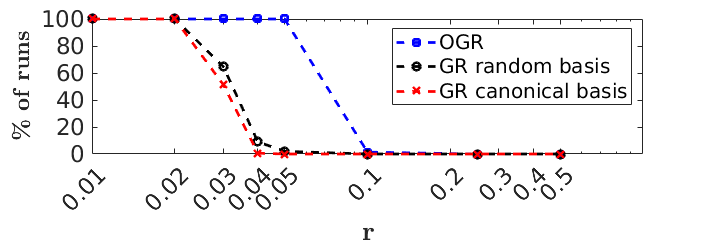}
  %\caption{A subfigure}
  \label{fig:BilinearSysLinearizedSetting2}
\end{subfigure}
\vspace*{-4mm}
\caption{Percentage of runs that converged to $\mu_\star \in\mbbc^{3\times 3}$ initialized by randomly chosen vectors on a %nine-dimensional 
sphere of radius $r$, for controls generated by LGR and OLGR for $1\%$ (top left) and $10\%$ (bottom left) relative error between $\mu_\star$ and $\mu_\circ$, and GR and OGR Algorithm \ref{algo:GeneralNonlin} with (bottom right) and without the shift by $\mu_\circ$ (top right).}
\label{fig:BilinearSysLinearizedSetting}
\end{figure}
We observe that the radii, up to which the control sets make GN capable of reliably reconstructing the global minimum, are much smaller than for the linear setting in Section \ref{subsec:NumLinear}.
When the initial relative error between $\mu_\circ=\mu(\al_\circ)$ and $\mu_\star=\mu(\al_\star)$ is very small ($1\%$) then LGR and OLGR have the most stable performance regarding the choice of the basis, making GN capable of reliably reconstructing the global minimum $\mu_\star$ up to a relative error of $4-5\%$.
However, when the initial relative error is larger ($10\%$) then only the LGR algorithm for the random basis can keep its performance, while even OLGR fails at errors of over $1\%$.
The results for OGR, on the other hand, show the best performance, with and without a shift by $\mu_\circ$.
The controls generated by the GR algorithms can not match OGR or LGR and OLGR for small initial errors, but are still more stable with respect to larger initial errors.

\subsection{\comment{Multi-spin (/-qubit) reconstruction problem}}
\comment{To investigate the performance of the algorithms for systems of larger dimensions, we consider the case of a $3$-qubit system. 
We use the following $64$-dimensional real representation of the Liouville master equation (compare, e.g., \cite[Section 2.12.1]{LibroQuantum}):
	\begin{equation*}
	\dot{\ypmb}(t)=\Biggl(A_\star+u(t)\sum_{n=1}^{6}B_n \Biggr)\ypmb(t),\quad t\in[0,T],\quad \ypmb(0)=\ypmb^0,
	\end{equation*}
    \vspace{-2mm}
	where $A_\star=2\pi\sum_{i=1}^{3}\omega_{\star,i}\widehat{A}_i+2\pi\sum_{j=1}^{2}J_{\star,j}\widetilde{A}_j\in\mbbr^{64\times64}$, $B_n=2\pi \widehat{B}_{\frac{n+1}{2}}$ for $n$ odd and $B_n=2\pi\widetilde{B}_{\frac{n}{2}}$ for $n$ even.
	Here, the matrices $\widehat{A}_i$ and $\widetilde{A}_j$ correspond to the free evolution and the coupling of the spins, and the matrices $\widehat{B}_n$ and $\widetilde{B}_n$ correspond to the controlled evolution.
	The goal is to reconstruct $A_\star$, i.e. the coefficients $\omega_{\star,i},J_{\star,j}$, which we assume to be equal to $1$.
	Thus, we define the set $\mathcal{A}=\{A_1,\ldots,A_{2N_p-1}\}$ as the union of the matrices $\widehat{A}_i$ and $\widetilde{A}_j$.
	Experimentally, we assume that the final state $\ypmb(T)$ is measured against a fixed state $\ypmb^1$, meaning that is $C=(\ypmb^1)^\top$.
    The vector $\ypmb^1$ consists of zeros except for the second, fifth and 17th entries which are equal to one, while the initial state $\ypmb^0$ is equal to one on the fourth, 13th and 49th entries.\footnote{All matrices $\widehat{A}_i$, $\widetilde{A}_j$, $\widehat{B}_n$ and $\widetilde{B}_n$ and the vectors $\ypmb^1$ and $\ypmb^0$ are generated from the codes of \cite{CIARAMELLA2015213}.}
	We perform the same experiments as in Section 8.1 with the difference that GR and OGR are run with the same basis $\mathcal{A}$.
	The results are reported in Figure \ref{fig:3Spins}.
	\begin{figure}[t]
    \centering
    \begin{subfigure}{.5\textwidth}
    \centering
    \includegraphics[width=0.9\textwidth]{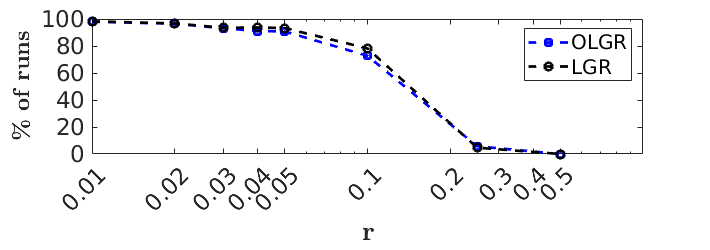}\\
    \includegraphics[width=0.9\textwidth]{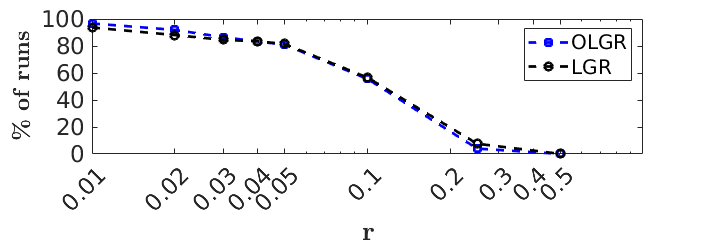}
    %\caption{A subfigure}
    %\label{fig:BilinearSysLinearizedSetting1}
    \end{subfigure}%
    \begin{subfigure}{.5\textwidth}
    \centering
    \includegraphics[width=0.9\textwidth]{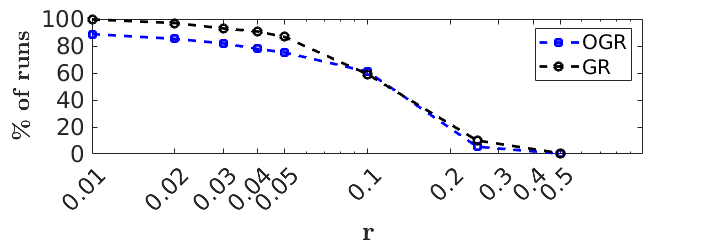}\\
    \includegraphics[width=0.9\textwidth]{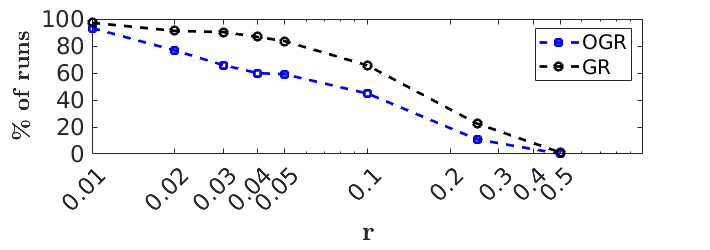}
    %\caption{A subfigure}
    %\label{fig:BilinearSysLinearizedSetting2}
    \end{subfigure}
    \vspace*{-4mm}
    \caption{Percentage of runs that converged to $A_\star$ initialized by randomly chosen vectors on a %nine-dimensional 
    sphere of radius $r$, for controls generated by LGR and OLGR for $1\%$ (top left) and $10\%$ (bottom left) relative error between $A_\star$ and $A_\circ$, and GR and OGR with (bottom right) and without the shift by $A_\circ$ (top right).}
    \label{fig:3Spins}
    \end{figure}
All algorithms show essentially a similar performance.
The results are similar to the ones of Section \ref{subsec:NumBilinear} (where the system dimensions are much smaller), showing that the performance of our algorithms does not deteriorates with the dimension of the system.}

%%\vspace*{-2mm}
%\section{Conclusion}\label{sec:conclusion}
%In this paper, we developed and analyzed greedy reconstruction algorithms based on the strategy presented in \cite{madaysalomon}.
%In particular, we tackled the case of nonlinear problems consisting in the reconstruction of drift operators in linear and bilinear dynamical systems. In these cases, we proved that the controls obtained by GR on the corresponding linearized systems lead to the local convergence of the classical Gauss-Newton method applied to the online nonlinear identification problem. These results were extended to the controls obtained on the fully nonlinear system (without linearization) where a local convergence result was also obtained.

\vspace*{0mm}
\section*{Acknowledgements}
G. Ciaramella is member of GNCS Indam.
The present research is part of the activities of “Dipartimento di Eccellenza 2023-2027”.
Simon Buchwald is funded by the DFG via the collaborative research center SFB1432, Project-ID 425217212. Julien Salomon acknowledges support from the 
ANR/RGC ALLOWAPP project (ANR-19-CE46-0013/A-HKBU203/19).

	\vspace*{-5mm}
	\bibliographystyle{siamplain}
	\bibliography{references}

\begin{thebibliography}{10}

\bibitem{Alexanderian2016}
{\sc A.~Alexanderian, N.~Petra, G.~Stadler, and O.~Ghattas}, {\em A fast and
  scalable method for {A}-optimal design of experiments for
  infinite-dimensional {B}ayesian nonlinear inverse problems}, SIAM J. Sci.
  Comput., 38 (2016), pp.~A243--A272.

\bibitem{aliprantis06}
{\sc C.~D. Aliprantis and K.~C. Border}, {\em Infinite Dimensional Analysis: a
  Hitchhiker's Guide}, Springer, 2006.

\bibitem{linearsystems}
{\sc P.~J. Antsaklis and A.~N. Michel}, {\em Linear Systems}, Birkhäuser
  Basel, 2006.

\bibitem{Atkinson2011}
{\sc A.~C. Atkinson}, {\em Optimum Experimental Design}, Springer Berlin
  Heidelberg, 2011.

\bibitem{Badouin2008}
{\sc L.~Baudouin and A.~Mercado}, {\em An inverse problem for {Schr\"odinger}
  equations with discontinuous main coefficient}, Appl. Anal., 87 (2008),
  pp.~1145--1165.

\bibitem{BAUER1960}
{\sc F.~Bauer and C.~Fike}, {\em Norms and exclusion theorems}, Numer. Math., 2
  (1960), pp.~137--141.

\bibitem{Bauer}
{\sc I.~Bauer, H.~G. Bock, S.~K{\"o}rkel, and J.~P. Schl{\"o}der}, {\em
  Numerical methods for optimum experimental design in dae systems}, J. Comput.
  Appl. Math., 120 (2000), pp.~1 -- 25.

\bibitem{nla.cat-vn740512}
{\sc C.~Berge}, {\em Topological spaces: including a treatment of multi-valued
  functions, vector spaces and convexity; Translated by E.M. Patterson}, Oliver
  \& Boyd Edinburgh, 1963.

\bibitem{MatrixExponentials}
{\sc D.~Bernstein and W.~So}, {\em Some explicit formulas for the matrix
  exponential}, IEEE Transactions on Automatic Control, 38 (1993),
  pp.~1228--1232, \url{https://doi.org/10.1109/9.233156}.

\bibitem{Bonnabel2009}
{\sc S.~Bonnabel, M.~Mirrahimi, and P.~Rouchon}, {\em Observer-based
  {H}amiltonian identification for quantum systems}, Automatica, 45 (2009),
  pp.~1144 -- 1155.

\bibitem{LibroQuantum}
{\sc A.~Borz\`i, G.~Ciaramella, and M.~Sprengel}, {\em Formulation and
  Numerical Solution of Quantum Control Problems}, SIAM, Philadelphia, PA,
  2017.

\bibitem{siampaper}
{\sc S.~Buchwald, G.~Ciaramella, and J.~Salomon}, {\em Analysis of a greedy
  reconstruction algorithm}, SIAM J. Control Optim., 59 (2021), pp.~4511--4537.

\bibitem{spinpaper}
{\sc S.~Buchwald, G.~Ciaramella, J.~Salomon, and D.~Sugny}, {\em Greedy
  reconstruction algorithm for the identification of spin distribution}, Phys.
  Rev. A, 104 (2021), p.~063112.

\bibitem{CIARAMELLA2015213}
{\sc G.~Ciaramella and A.~Borzì}, {\em {SKRYN: A fast
  semismooth-Krylov–Newton method for controlling Ising spin systems}},
  Computer Physics Communications, 190 (2015), pp.~213--223.

\bibitem{Ciarlet}
{\sc P.~G. Ciarlet}, {\em Linear and Nonlinear Functional Analysis with
  Applications}, Applied mathematics, SIAM, Philadelphia, PA, 2013.

\bibitem{coron2007control}
{\sc J.~Coron}, {\em Control and Nonlinearity}, Math. Surveys Monogr., AMS,
  2007.

\bibitem{DeSchutter2000}
{\sc B.~De~Schutter}, {\em Minimal state-space realization in linear system
  theory: An overview}, J. Comput. Appl. Math., 121 (2000), pp.~331--354.

\bibitem{Donovan2014}
{\sc A.~Donovan and H.~Rabitz}, {\em Exploring the {H}amiltonian inversion
  landscape}, Phys. Chem., 16 (2014), pp.~15615--15622.

\bibitem{Fu_2017}
{\sc Y.~Fu and G.~Turinici}, {\em Quantum {H}amiltonian and dipole moment
  identification in presence of large control perturbations}, ESAIM: Contr.
  Optim. Ca., 23 (2017), pp.~1129--1143.

\bibitem{Geremia2001}
{\sc J.~M. Geremia and H.~Rabitz}, {\em Global, nonlinear algorithm for
  inverting quantum-mechanical observations}, Phys. Rev. A, 64 (2001),
  p.~022710.

\bibitem{Geremia2003}
{\sc J.~M. Geremia and H.~Rabitz}, {\em Optimal {H}amiltonian identification:
  The synthesis of quantum optimal control and quantum inversion}, J. Chem.
  Phys., 118 (2003), pp.~5369--5382.

\bibitem{Geremia2000}
{\sc J.~M. Geremia, W.~Zhu, and H.~Rabitz}, {\em Incorporating physical
  implementation concerns into closed loop quantum control experiments}, J.
  Chem. Phys., 113 (2000), pp.~10841--10848.

\bibitem{glaser_training_2015}
{\sc S.~Glaser and et~al}, {\em Training {S}chrödinger’s cat: quantum
  optimal control}, Eur. Phys. J. D, 69 (2015), p.~279.

\bibitem{horn_johnson_1991}
{\sc R.~A. Horn and C.~R. Johnson}, {\em Topics in Matrix Analysis}, Cambridge
  Univ. Press, 1991.

\bibitem{Juang1985}
{\sc J.~N. Juang and R.~S. Pappa}, {\em An eigensystem realization algorithm
  for modal parameter identification and model reduction}, J. Guid. Control
  Dynam., 8 (1985), pp.~620--627.

\bibitem{kaltenbacher2008iterative}
{\sc B.~Kaltenbacher, A.~Neubauer, and O.~Scherzer}, {\em Iterative
  Regularization Methods for Nonlinear Ill-Posed Problems}, De Gruyter, Berlin,
  New York, 2008.

\bibitem{kelley}
{\sc C.~T. Kelley}, {\em Iterative Methods for Optimization}, SIAM,
  Philadelphia, 1999.

\bibitem{lebris:ham_id}
{\sc C.~Le~Bris, M.~Mirrahimi, H.~Rabitz, and G.~Turinici}, {\em {{H}amiltonian
  identification for quantum systems: Well posedness and numerical
  approaches}}, {ESAIM: Contr. Optim. Ca.}, 13 (2007), pp.~378--395.

\bibitem{madaysalomon}
{\sc Y.~Maday and J.~Salomon}, {\em {A greedy algorithm for the identification
  of quantum systems}}, in {Proceedings of the 48th IEEE Conference on Decision
  and Control}, {}, {2009}, pp.~{375--379}.

\bibitem{rudin}
{\sc W.~Rudin}, {\em Real and Complex Analysis, 3rd Ed.}, McGraw-Hill, Inc.,
  USA, 1987.

\bibitem{Sontag1998}
{\sc E.~D. Sontag}, {\em Mathematical Control Theory: Deterministic Finite
  Dimensional Systems (2Nd Ed.)}, Springer-Verlag, Berlin, Heidelberg, 1998.

\bibitem{Wang}
{\sc Y.~{Wang}, D.~{Dong}, B.~{Qi}, J.~{Zhang}, I.~R. {Petersen}, and
  H.~{Yonezawa}}, {\em A quantum {H}amiltonian identification algorithm:
  Computational complexity and error analysis}, IEEE Trans. Autom. Control, 63
  (2018), pp.~1388--1403.

\bibitem{Whittlesey}
{\sc E.~F. Whittlesey}, {\em Analytic functions in {B}anach spaces},
  Proceedings of the American Mathematical Society, 16 (1965), pp.~1077--1083.

\bibitem{Xue2019}
{\sc S.~Xue, R.~Wu, D.~Li, and M.~Jiang}, {\em A gradient algorithm for
  {H}amiltonian identification of open quantum systems}, Phys. Rev. A, 103
  (2021), p.~022604.

\bibitem{Zhang2014}
{\sc J.~Zhang and M.~Sarovar}, {\em Quantum {H}amiltonian identification from
  measurement time traces}, Phys. Rev. Lett., 113 (2014), p.~080401.

\bibitem{Zhou2012}
{\sc W.~Zhou, S.~Schirmer, E.~Gong, H.~Xie, and M.~Zhang}, {\em Identification
  of {Markovian} open system dynamics for qubit systems}, Chinese Sci. Bull.,
  57 (2012), pp.~2242--2246.

\end{thebibliography}
	
\end{document}